\documentclass[11pt,a4paper]{article} 

\usepackage{amsfonts}
\usepackage{amssymb}
\usepackage{amsmath}
\usepackage{amsthm}
\usepackage{authblk}
\usepackage{array}
\usepackage{bbm}
\usepackage{booktabs}
\usepackage{dsfont}
\usepackage{enumerate}
\usepackage{graphics}
\usepackage{graphicx}
\usepackage{hyperref}
\usepackage{inputenc}
\usepackage{latexsym}
\usepackage{mathrsfs}  
\usepackage{tikz}
\usepackage[normalem]{ulem}

\newtheorem{theorem}{Theorem}[section]
\newtheorem{proposition}[theorem]{Proposition}
\newtheorem{lemma}[theorem]{Lemma}
\newtheorem{corollary}[theorem]{Corollary}

\theoremstyle{definition}
\newtheorem{definition}{Definition}[section]
\newcommand*\circled[1]{\tikz[baseline=(char.base)]{\node[shape=circle,draw,inner sep=.9pt] (char) {#1};}}
\newcommand{\vge}{\rotatebox[origin=c]{-90}{$\ge$}}
\newcommand{\vle}{\rotatebox[origin=c]{-90}{$\le$}}

\newcommand{\vdequal}{\rotatebox[origin=c]{-90}{$=:$}}

\title{Weighted Quasi Interpolant Spline Approximations: \\ Properties and Applications}
\author[1,2,3]{Andrea Raffo}
\author[3]{Silvia Biasotti}
\affil[1]{Department of Mathematics and Cybernetics, SINTEF, Oslo, Norway.}
\affil[2]{Department of Mathematics, University of Oslo, Oslo, Norway.}
\affil[3]{Istituto di Matematica Applicata e Tecnologie Informatiche ``E. Magenes'' CNR, Genova, Italy.}
\date{}                     
\begin{document}
\maketitle
\begin{abstract}
Continuous representations are fundamental for modeling sampled data and performing computations and numerical simulations directly on the model or its elements. To effectively and efficiently address the approximation of point clouds we propose the Weighted Quasi Interpolant Spline Approximation method (wQISA). We provide global and local bounds of the method and discuss how it still preserves the shape properties of the classical quasi-interpolation scheme. This approach is particularly useful when the data noise can be represented as a probabilistic distribution: from the point of view of nonparametric regression, the wQISA estimator is robust to random perturbations, such as noise and outliers. Finally, we show the effectiveness of the method with several numerical simulations on real data, including curve fitting on images, surface approximation and simulation of rainfall precipitations.  
\end{abstract}

\section{Introduction}

Modelling sampled data with a continuous representation is essential in many applications such as, for instance, image resampling \cite{splineimages}, geometric modelling \cite{Farin:1993}, isogeometric analysis (IgA) \cite{hughes} and the numerical solution of PDE boundary problems \cite{Boffi2013}.

Spline interpolation is largely adopted to approximate data from a function or a physical object because of the simplicity of its construction, its ease and accuracy of evaluation, and its capacity to approximate complex shapes through mathematical element fitting and interactive design \cite{Schumaker2007}.
It is often preferred to polynomial interpolation because it yields visually effective results even when using low degree polynomials, while avoiding the Runge's phenomenon for higher degrees \cite{Gregory:1986}. B-splines represent a popular way for dealing with spline interpolation and are nowadays the most powerful tool in CAGD \cite{Buffa2012}. Several generalizations to non-polynomial splines are possible, such as generalized splines \cite{BRACCO2016}, which admit also trigonometric or exponential bases, or non-uniform rational B-splines (NURBS) \cite{nurbs}. The B-spline extension to higher dimensions consists of multivariate spline functions based on a tensor product approach. Unfortunately, classical tensor product splines lack local refinement, which is often fundamental in those applications dealing with large amounts of data. For this reason several alternative structures that support local refinement have been introduced in the last decades; for instance, in the context of a tensor-product paradigm, T-splines \cite{Sederberg2003}, hierarchical B-splines \cite{Forsey1988}, locally refined (LR) B-splines \cite{Dokken2013} and (truncated) Hierarchical B-splines (THB) \cite{Giannelli2016}.

When dealing with real data -- for instance, acquired by laser scanners, photogrammetry and diagnostic devices -- there are many source of uncertainty, such as resolution, precision, occlusions and reflections. Furthermore, digital models often undergo post-processing stages after acquisition, and these may introduce additional geometric and/or numerical artefacts \cite{Cao:2017}.
Most of the existing inverse approximation techniques are executed as a deterministic problem and the parameters involved in the model are treated as unambiguous values. Despite the recent introduction of uncertainty-based inverse analysis tools such as evidence-theory, fuzzy and interval uncertainties \cite{long2019}, at the best of our knowledge, only few  modelling approaches identify uncertainty theories as a good solution for explicitly modelling data uncertainty, adopting, for instance, interproximation \cite{CHENG1991} or fuzzy numbers \cite{ANILE2000}. Unfortunately, these efforts were quite isolated and their computational complexity prevented their massive adoption.

In this scenario, we aim at preserving the use of B-spline bases because of their simplicity, their approximation capability and accuracy. To effectively and efficiently approximate raw data and point clouds possibly affected by noise and outliers, we propose the adoption of a novel quasi-interpolation scheme. Quasi-interpolation is a well known technique \cite{deboor73,Sablonniere05} that does not require to solve any linear system, unlike the traditional spline approaches, and therefore it allows to define more efficient algorithms. Whilst there are works on the use of quasi-interpolant methods for function approximation  \cite{deboor73,Buhmann:1993,Buhmann:2015,JIANG2011,Patrizi2020,Speleers2016}, to the best of our knowledge, less efforts have been devoted to define quasi-interpolant schemes for point clouds \cite{Amir2018,Beatson:1992,BRACCO2017,Gao2018}.

As working assumptions, we assume the point cloud to be embedded in an Euclidean space $\mathbb{R}^{d+1}$ and locally represented as a height field $y=f(x_1,\dots,x_d) $. We obtain a method which is not only robust, but also has a reduced computational complexity thanks to the adopted quasi-interpolation scheme. The method properties, presented in detail for the uni- and bivariate cases for simplicity of notation, can be easily extended to consider data of arbitrary dimension. We also discuss how the shape properties of monotonicity and convexity derive from classical spline theory. Since we aim at addressing data affected by noise, we provide a probabilistic interpretation of the method. We illustrate its properties over a number of examples, ranging from curve fitting to the approximation of scalar fields defined on surfaces.  In summary, the main contributions of this work are:  
\begin{itemize}
    \item The introduction of a novel quasi-interpolation scheme to approximate point clouds, possibly affected by noise and outliers, together with a theoretical study of its numerical properties (Section \ref{numerical_formulation}).
    \item The interpretation of our approach in terms of the nonparametric regression scheme, together with the theoretical study of bias and variance of the wQISA estimator (Section \ref{sec:prob_interp}).
    \item The validation of the method on real data from different applications, including curve fitting, surface reconstruction and rainfall approximation and forecasting (Section \ref{sec:simulations}).
\end{itemize}
Finally, concluding remarks are provided in Section \ref{sec:conclusion}.

\section{Weighted quasi-interpolant spline approximation for point clouds\label{numerical_formulation}}

In this Section we first summarise some basic notation and definitions on B-splines. We then formally introduce the weighted quasi-interpolant spline approximations, provide their global and local bounds and discuss in what sense they preserve the shape properties.

\subsection{Basic concepts on spline spaces}
\label{TPBSplines}
From B-splines theory, it is well known that a non-decreasing sequence $\textbf{t}=[t_1,\dots,t_{n+p+1}]$, which is commonly referred to as \emph{global knot vector}, generates $n$ B-splines of degree $p$ over $\textbf{t}$. In practice, the construction of each of these B-splines requires only a subsequence of $p+2$ consecutive knots, collected in a \textit{local knot vector}.

\begin{definition}[Univariate B-spline]
Let $\mathbf{t}:=[t_1,\dots,t_{p+2}]$ be a (local) knot vector. A \emph{B-spline} $B[\mathbf{t}]:\mathbb{R}\to\mathbb{R}$ of degree $p$ is the function recursively defined by
\begin{equation}
B[\mathbf{t}](x):=\dfrac{x-t_1}{t_{p+1}-t_1}B[t_1,\dots,t_{p+1}](x)+\dfrac{t_{p+2}-x}{t_{p+2}-t_2}B[t_2,\dots,t_{p+2}](x),
\label{eqn:UnivariateBspline}
\end{equation}
where 
\[
B[t_i,t_{i+1}](x):=
\begin{cases}
1, & \text{if $x\in [t_i,t_{i+1})$} \\
0, &\text{elsewhere}
\end{cases}
, \quad  i=1,\dots,p+1.
\]
Here, the convention is assumed that ``$0/0=0$''. 
\end{definition}

By assuming $t_1<t_{p+2}$, it follows that $B[\mathbf{t}]$ is a piecewise polynomial of degree $p$. The continuity at each unique knot is $p-m$, where $m$ is the number of times the knot is repeated. $B[\mathbf{t}]$ is smooth in each open subinterval $(t_i,t_{i+1})$, where $i=1,\dots,p+1$, and is non-negative over $\mathbb{R}$. The \emph{support} of $B[\mathbf{t}]$, i.e., the closure of the subset of the domain where $B[\mathbf{t}]$ is non-zero, is the compact interval $\text{supp}(B[\mathbf{t}])=[t_1,t_{p+2}]$. 
\begin{definition}[Univariate spline space]
Given a global knot vector $\mathbf{t}=[t_1,\dots,t_{n+p+1}]$, the \emph{spline space} $\mathbb{S}_{p,\mathbf{t}}$ is the linear space defined by
\[
\mathbb{S}_{p,\mathbf{t}}:=\text{span}\left\{B[\mathbf{t}^{(1)}],\dots,B[\mathbf{t}^{(n)}]\right\},
\]
where $\mathbf{t}^{(i)}:=[t_i,\ldots,t_{i+p+1}]$ for any $i=1,\dots,n$. An element $f\in\mathbb{S}_{p,\mathbf{t}}$ is called a \emph{spline function}, or just a \emph{spline}, of degree $p$ with knots $\mathbf{t}$.
\end{definition}
By assuming that no knot occurs more than $p+1$ times, it follows that $\left\{B[\mathbf{t}^{(i)}]\right\}_{i=1}^n$ is a basis for $\mathbb{S}_{p,\mathbf{t}}$. A B-spline basis forms a partition of unity over $[t_1,t_{n+p+1}]$. We can refine a spline curve $f=\sum_{i=1}^nb_iB[\mathbf{t}^{(i)}]$ by inserting new knots in $\mathbf{t}$ and then computing the coefficients of $f$ in the augmented spline space. An efficient way to perform this process is the \emph{Oslo algorithm} \cite{Cohen1980}. 

Lastly, we specify the type of knot vectors we will consider in the next sections, as they allow to define B-spline bases that interpolate the boundaries. 

\begin{definition}
\label{definition:regularknotvector}
A knot vector $\mathbf{t}=[t_1,\ldots,t_{n+p+1}]$ is said to be $(p+1)$\emph{-regular} if 
\begin{enumerate}[1.]
\item $n\ge p+1$,
\item $t_1=t_{p+1}$ and $t_{n+1}=t_{n+p+1}$,
\item $t_j<t_{j+p+1}$ for $j=1,\ldots,n$.
\end{enumerate}
\end{definition}

\begin{definition}[Tensor product B-spline]
A \emph{tensor product B-spline} of multi-degree $\mathbf{p}:=(p_1,\ldots,p_d)\in\mathbb{N}^d$ is a separable function $B:\mathbb{R}^d\to \mathbb{R}$ defined as
\begin{equation}
\label{eqn:DefTensorProductBSpline}
B[\mathbf{t}_1,\ldots,\mathbf{t}_d](\mathbf{x}):=\prod_{k=1}^dB[\mathbf{t}_k](x_k),
\end{equation}
where $\mathbf{x}=(x_1,\ldots,x_d)$ and $\mathbf{t}_k=[t_{k,1},\ldots,t_{k,p_k+2}]\in\mathbb{R}^{p_k+2}$ is the local knot vector along $x_k$, for any $k=1,\ldots,d$.
\end{definition}
By assuming that $t_{k,1}<t_{k,p_k+2}$ for any $k=1,\ldots,d$, it follows that $B[\mathbf{t}_1,\ldots,\mathbf{t}_d]$ is a piecewise polynomial of multi-degree $\mathbf{p}$. 

\begin{definition}[Tensor product spline space]
A \emph{tensor product spline space} $\mathbb{S}_{\mathbf{p},[\mathbf{t}_1,\ldots,\mathbf{t}_d]}$ is the linear space defined by
\[
\mathbb{S}_{\mathbf{p},[\mathbf{t}_1,\ldots,\mathbf{t}_d]}:=
\bigotimes_{k=1}^{d}\mathbb{S}_{p_k,\mathbf{t}_k}=
\text{span}\left\{\prod_{k=1}^dB[\mathbf{t}_k^{(i_k)}] \text{ s.t. }i_k=1,\ldots,n_k\right\},
\]
where $\mathbf{t}_k\in\mathbb{R}^{n_k+p_k+1}$ is a global knot vector for any $k=1,\ldots,d$. An element $f\in\mathbb{S}_{\mathbf{p},[\mathbf{t}_1,\ldots,\mathbf{t}_d]}$ is called a \emph{tensor product spline} function, or just a \emph{spline}, of multi-degree $\mathbf{p}$ with knot vectors $\mathbf{t}_1,\ldots,\mathbf{t}_d$.
\end{definition}
The tensor product spline representation inherits all the properties (local support, non-negativity, local smoothness, partition of unity) of the univariate case. We refer the reader to \cite{Schumaker2007} for a more exhaustive introduction to B-splines.

\subsection{Weighted Quasi Interpolation Spline Approximation}
We introduce our method for the general case of a point cloud $\mathcal{P}\subset\mathbb{R}^{d+1}$. Again, we assume that the point cloud can be locally represented by means of a function 
\[
y=f(x_1,\dots,x_d).
\]

\begin{definition}
\label{definition:wQISA}
Let $\mathcal{P}\subset\mathbb{R}^{d+1}$ be a point cloud and $\mathbf{p}\in\mathbb{N}^d$ a (multi)-degree with all nonzero components. Let $\mathbf{t}_k\in\mathbb{R}^{n_k+p_k+1}$ be a $(p_k+1)$-regular knot vector with boundary knots $t_{p_k}=a_k$ and $t_{n_k+1}=b_k$, for $k=1,\ldots,d$. The \emph{Weighted Quasi Interpolant Spline Approximation}  (wQISA) of degree $\mathbf{p}$ to the point cloud $\mathcal{P}$ over the knot vectors $\mathbf{t}_k$ is defined by
\begin{equation}
\label{equation:dkNNVDSA}
f_w:=\sum_{i_1=1}^{n_1}\ldots\sum_{i_d=1}^{n_d}\hat{y}_w\left(\xi_1^{(i_1)},\ldots,\xi_d^{(i_d)}\right)\cdot B[\mathbf{t}_{1}^{(i_1)},\ldots,\mathbf{t}_{d}^{(i_d)}],
\end{equation}
where $\xi_k^{(i_k)}:=(t_{k,i_k+1}+\ldots+t_{k,i_k+p_k})/p_k$ are the \emph{knot averages} and 
\begin{equation}
\label{equation:control_points_estimator}
    \hat{y}_w(\mathbf{u}):=\dfrac{\sum\limits_{(x_1,\ldots,x_d,y)\in\mathcal{P}}y\cdot w_{\mathbf{u}}(x_1,\ldots,x_d)}{\sum\limits_{(x_1,\ldots,x_d,y)\in\mathcal{P}}w_{\mathbf{u}}(x_1,\ldots,x_d)}
\end{equation}
are the \emph{control points estimators} of \emph{weight functions} $w_{\mathbf{u}}:\mathbb{R}^d\to[0,+\infty)$.
\end{definition}

The function $w_{\mathbf{u}}:\mathbb{R}^d\to[0,+\infty)$ of Definition \ref{definition:wQISA} defines a \emph{window} around each point $\mathbf{u}\in\mathbb{R}^d$ and is also called a \emph{Parzen window}. An example is the weight function: 
\begin{equation}
w_{\mathbf{u}}(\mathbf{x}):=
\begin{cases}
1/k, 	& \text{ if } \mathbf{x}\in N_k(\mathbf{u}) \\
0,	& \text{ otherwise }
\end{cases},
\label{equation:kNN_weight}
\end{equation}
where $k\in\mathbb{N}^\ast$ and $N_k(\mathbf{u})$ denotes the neighborhood of $\mathbf{u}$ defined by the $k$ closest points of the point cloud. In this case, $\hat{y}_w$ defines the $k$-nearest neighbor ($k$-NN) regressor (see figure \ref{figure:kNN_weight}). Commonly, the function $w_{\mathbf{u}}$ depends on a distance, for examples:
\begin{subequations}
\begin{align}
	& w_{\mathbf{u}}(\mathbf{x})=\,\,\,\mathbbm{1}_{||\mathbf{x}-\mathbf{u}||_2\le r}					& (\text{Characteristic})		\label{eq:char_weight} 	\\
	& w_{\mathbf{u}}(\mathbf{x})=e^{-||\mathbf{x}-\mathbf{u}||_2/2\sigma^2} 						& (\text{Gaussian})			\label{eq:gauss_weight} 	\\
	& w_{\mathbf{u}}(\mathbf{x})=e^{-||\mathbf{x}-\mathbf{u}||_2/\sqrt{2}\sigma}					& (\text{Exponential})			\label{eq:exp_weight}
\end{align}
\end{subequations}

\begin{figure}[htp]
\centering
\includegraphics[scale=0.33]{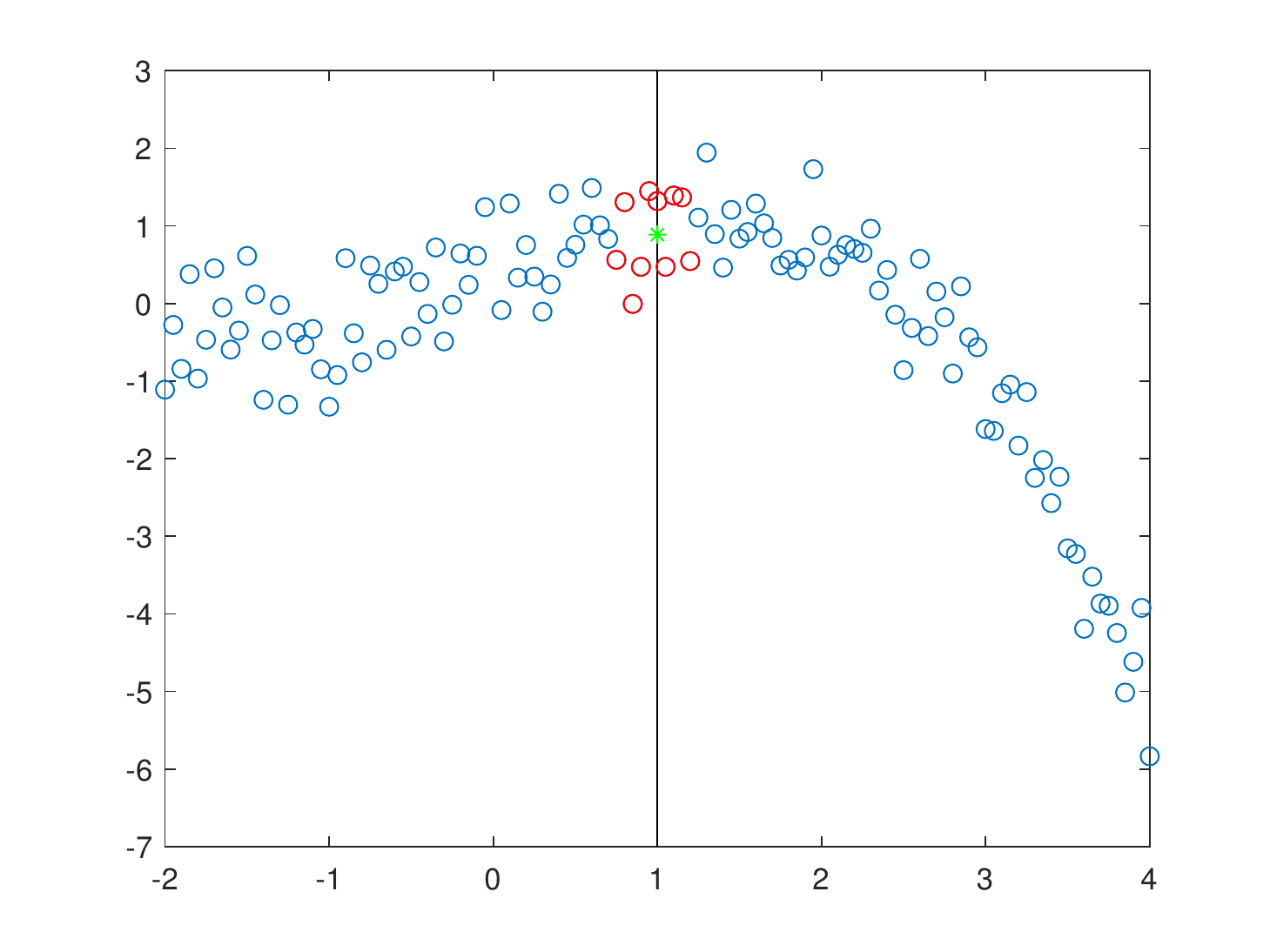}
\caption{Parzen windows and control points estimators. Given a 2D point cloud (in blue), we compute $\hat{y}_w$ at $u=1$ (in green) by using the $10$ nearest points (in red).}
\label{figure:kNN_weight}
\end{figure}

Note that: 
\begin{itemize}
    \item $w_{\mathbf{u}}$ depends on the point $\mathbf{u}\in\mathbb{R}^d$ of interest, and can thus be adapted to local information (e.g., variable level and/or nature of noise).
    \item The quality of an approximation strongly depends on the spline space and the weight functions that are chosen in Definition \ref{definition:wQISA}. As shown in Figure \ref{figure:wQISA_defn}, a given spline space and weight function is not always able to capture the relevant trends of a point set.
    \begin{figure}[!ht]
    \begin{center}
    \begin{tabular}{ccc}
    \fbox{
    \includegraphics[scale=0.25, trim={1cm 0cm 1cm 0cm}, clip]{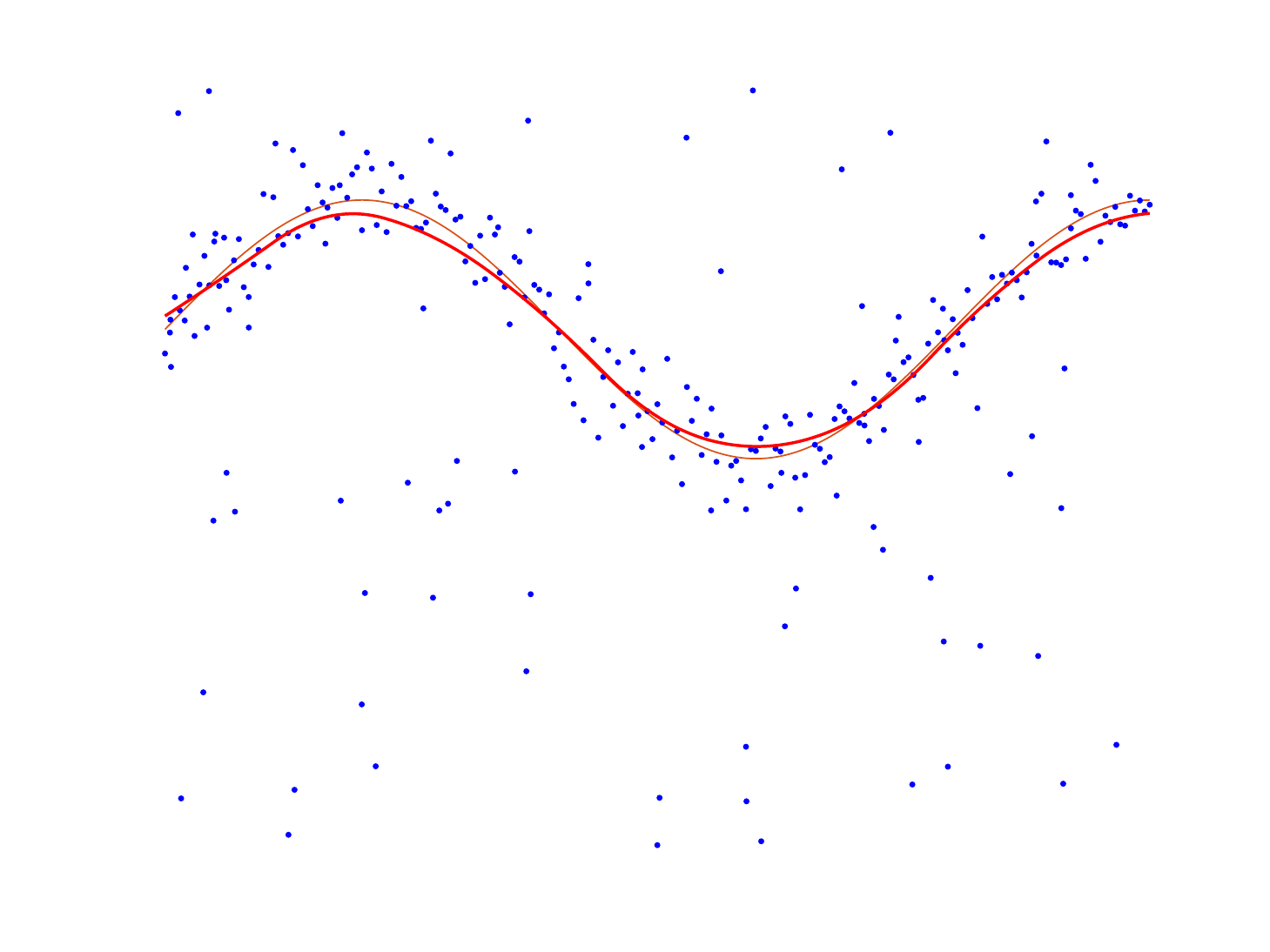}
    }
    &
    \fbox{
    \includegraphics[scale=0.25, trim={1cm 0cm 1cm 0cm}, clip]{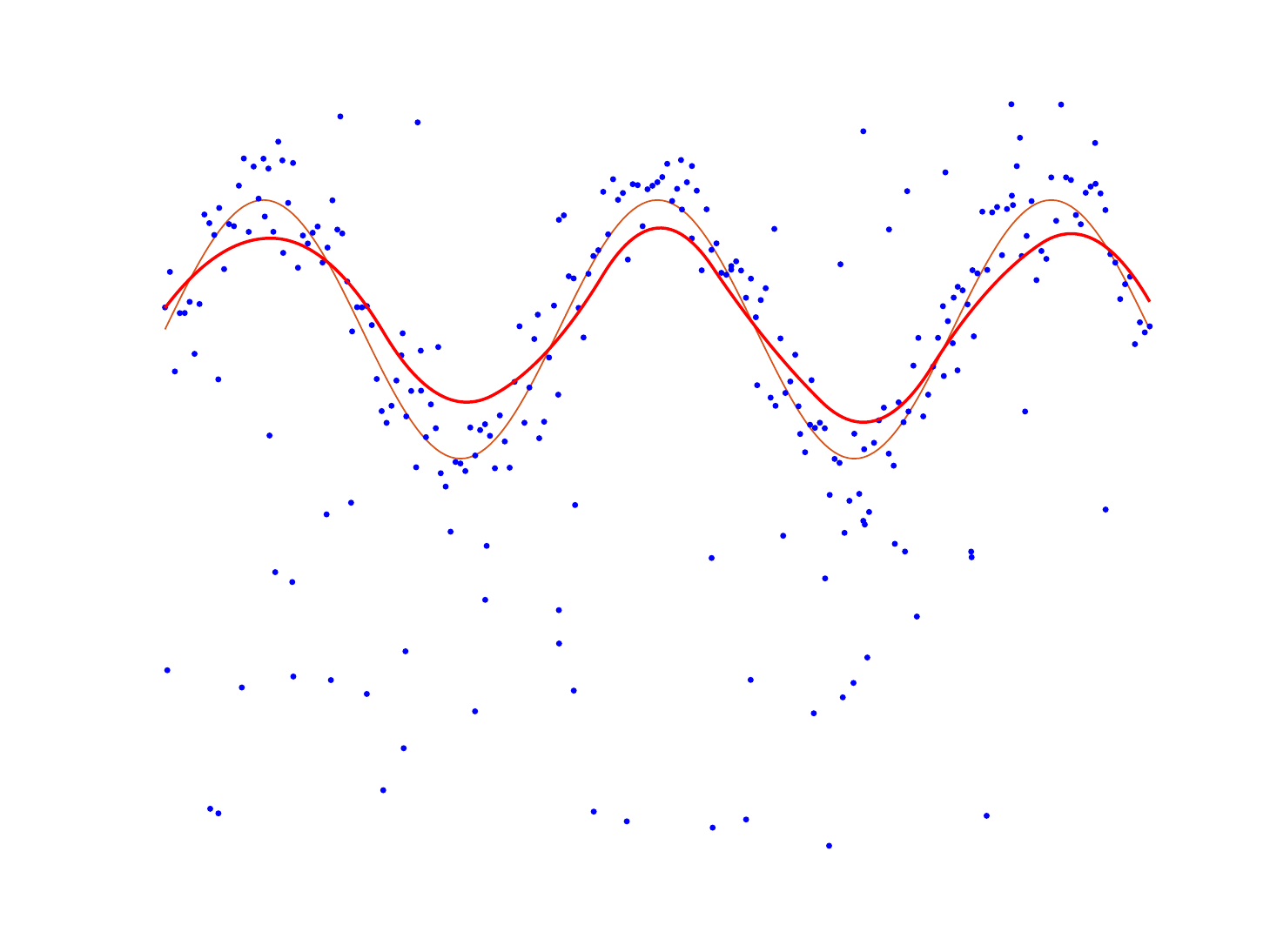}
    }
    &
    \fbox{
    \includegraphics[scale=0.25, trim={1cm 0cm 1cm 0cm}, clip]{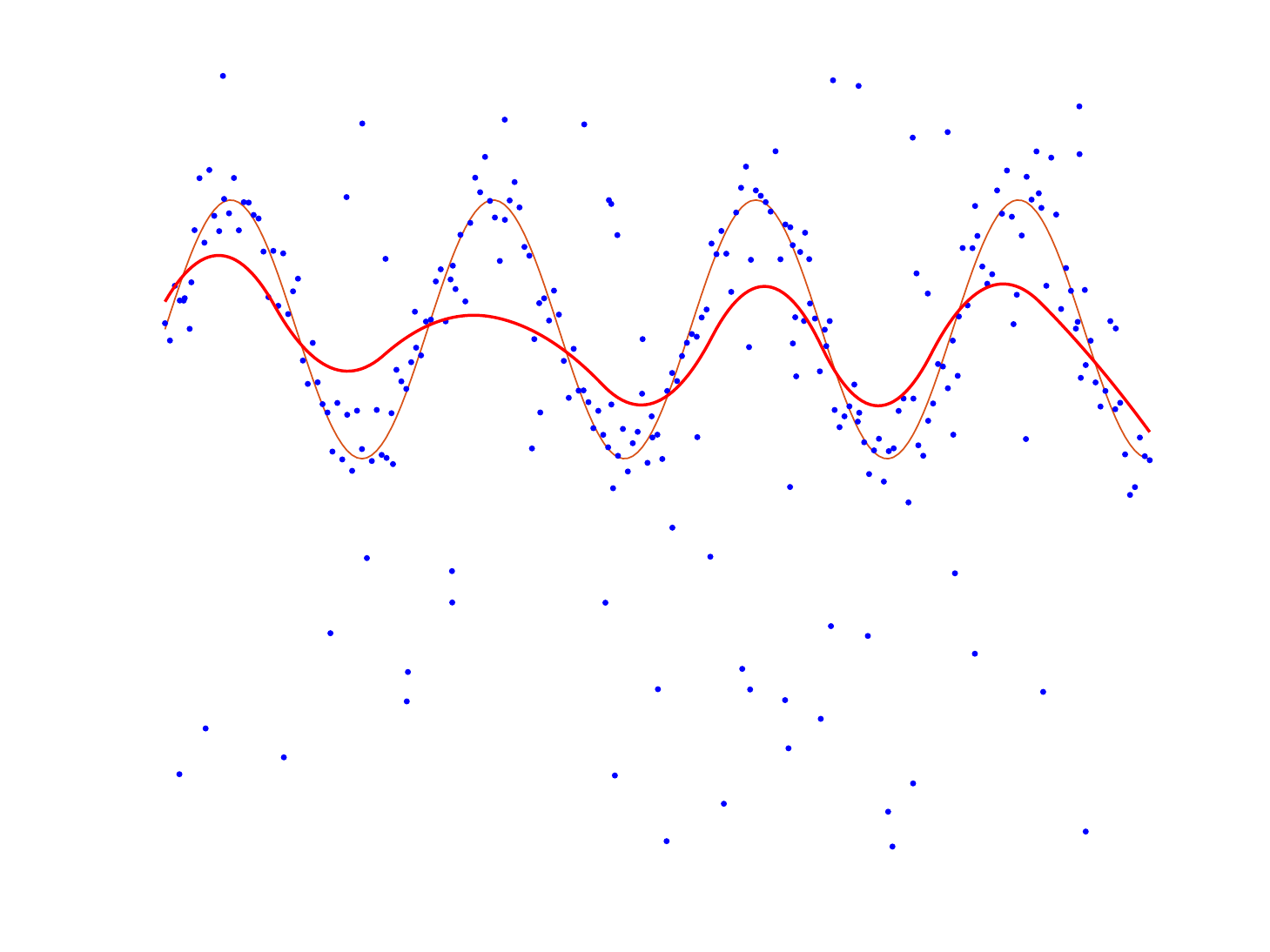}
    } \\
    (a) & (b) & (c)
    \end{tabular}
    \end{center}
    \caption{wQISA curve approximation of three point clouds. The point sets (in blue) are sampled from $y=sin(\pi x)$ in (a), $y=sin(2\pi x)$ in (b) and $y=sin(3\pi x)$ in (c) and then perturbed with Gaussian noise and outliers. Here, we consider a spline space of dimension $10$ over a uniform knot vector and a Gaussian weight function (see Equation \eqref{eq:gauss_weight}) of fixed variance, combined with quartiles to filter the outliers. The figures shows the original functions (in orange) and the approximations (in red). }
    \label{figure:wQISA_defn}
    \end{figure}
\end{itemize}

\subsection{Properties}
\label{sec:properties}
We first introduce bounds for the wQISA approximation. We then explain in what sense shape properties (monotonicity and convexity) are preserved in case of raw data. While we refer the reader to Appendix \ref{Appendix} for a detailed introduction of the univariate case, here, we focus our attention on the bivariate setting, i.e., on  representations of the form $z=f(x,y)$. The extension of these results to higher dimensions is straightforward and just requires a more involved notation.

For the sake of simplicity, we re-write Equation \eqref{equation:dkNNVDSA} as
\begin{equation*}
    f_w(x,y):=\sum_{i=1}^{n_x}\sum_{j=1}^{n_y}\hat{z}_w\left(\xi_x^{(i)},\xi_y^{(j)}\right)\cdot B[\mathbf{x}^{(i)},\mathbf{y}^{(j)}](x,y),
\end{equation*}
where we customize the notation by denoting with $\xi_x^{(i)}$ (resp. $\xi_y^{(j)}$) the $i$-th (resp. $j$-th)  knot average with respect to the global knot vector $\mathbf{x}$ (resp. $\mathbf{y}$) along x (resp. y).

\subsubsection{Global and local bounds}
\begin{proposition}[Global bounds]
\label{proposition:gb_multi}
Let $\mathcal{P}\subset\mathbb{R}^3$ be a point cloud. Given $z_{\text{min}}$, $z_{\text{max}}\in\mathbb{R}$  that satisfy
\[
z_{\text{min}}\le z\le z_{\text{max}}, \quad \text{ for all } (x,y,z)\in\mathcal{P},
\]
then the weighted quasi interpolant spline approximation to $\mathcal{P}$ from some spline space $\mathbb{S}_{p,[\mathbf{x},\mathbf{y}]}$ and some family of weight functions $w_u:\mathbb{R}^2\to[0,+\infty)$ has the same bounds
\[
z_{min}\le f_w(x,y)\le z_{max}, \quad \text{ for all } (x,y)\in\mathbb{R}^2.
\]
\end{proposition}
\begin{proof}
From the partition of unity of a B-spline basis, it follows that
\begin{equation}
\label{eqn:gb_multi}
\resizebox{0.99\hsize}{!}{
\begin{math}
\begin{array}{@{}c@{\;}c@{\;}c@{\;}c@{\;}c@{}}
\min\limits_i\hat{z}_w(\xi_x^{(i)},\xi_y^{(j)}) & \le & \sum\limits_{i=1}^{n_x}\sum\limits_{j=1}^{n_y}\hat{z}_w(\xi_x^{(i)},\xi_y^{(j)})\cdot B[\mathbf{x}^{(i)},\mathbf{y}^{(j)}] & \le & \max\limits_i\hat{z}_w(\xi_x^{(i)},\xi_y^{(j)}) \\
\vge\,\scriptsize{\circled{1}} && \vdequal  && \vle\,\scriptsize{\circled{2}} \\
z_{min} &&f_w && z_{max}
\end{array}
\end{math}
}
\end{equation}
where the inequalities $\circled{1}$ and $\circled{2}$ are a direct consequence of defining $\hat{z}_w$ by means of a convex combination.
\end{proof}
The bounds of Proposition \ref{proposition:gb_multi} can potentially lead to local bounds, for example when the weight functions have bounded support. We discuss this possibility in Corollary \ref{corollary:lb_multi}.

\begin{corollary}[Local bounds]
\label{corollary:lb_multi}
Let $\mathcal{P}\subset\mathbb{R}^3$ be a point cloud. Let $x\in[x_\mu,x_{\mu+1})$ for some $\mu$ in the range $p_x+1\le\mu\le n_x$ and $y\in[y_\nu,y_{\nu+1})$ for some $\nu$ in the range $p_y+1\le\nu\le n_y$. Then
\[
\alpha(\mu,\nu)\le f_w(x,y)\le\beta(\mu,\nu)
\]
for some $\alpha(\mu,\nu),\beta(\mu,\nu)$ which belong to $[z_{min},z_{max}]$.
\end{corollary}
\begin{proof}
By using the property of local support for B-splines, it follows that
\[
f_w(x,y)=\sum_{i=\mu-p_x}^\mu\sum_{j=\nu-p_y}^\nu\hat{z}_w(\xi_x^{(i)},\xi_y^{(j)})B[\mathbf{x}^{(i)},\mathbf{y}^{(j)}](x,y).
\]
Hence
\begin{equation}
\label{equation:local_bounds}
\begin{array}{@{}c@{\;}c@{\;}c@{\;}c@{\;}c@{}}
\min\limits_{\substack{i=\mu-p_x,\dots,\mu \\ j=\nu-p_y,\dots,\nu}}\hat{z}_w(\xi_x^{(i)},\xi_y^{(j)}) & \le & f_w(x,y) & \le & \max\limits_{\substack{i=\mu-p_x,\dots,\mu \\ j=\nu-p_y,\dots,\nu}}\hat{z}_w(\xi_x^{(i)},\xi_y^{(j)}) \\
\vge\,\scriptsize{\circled{3}} &&   && \vle\,\scriptsize{\circled{4}} \\
\min\Big\{z \text{ s.t. }(x,y,z)\in\mathcal{P}_{\mu,\nu}\Big\} && && \max\Big\{z \text{ s.t. }(x,y,z)\in\mathcal{P}_{\mu,\nu}\Big\} \\
\vdequal &&   && \vdequal \\
\alpha(\mu,\nu) && && \beta(\mu,\nu)
\end{array}
\end{equation}
where 
\[
\mathcal{P}_{\mu,\nu}:={\textstyle\bigcup\limits_{\substack{i=\mu-p_x,\dots,\mu \\ j=\nu-p_y,\dots,\nu}}}
\text{supp}\left(w_{(\xi_x^{(i)},\xi_y^{(j)})}\right)
\cap\mathcal{P}
\]
and where $\text{supp}$ denotes the support of a function. The inequalities $\circled{3}$ and $\circled{4}$ are a direct consequence of defining $\hat{z}_w$ by means of a convex combination. Note that the set $\mathcal{P}^\ast$ of points that are effectively used to compute the approximation, i.e.,
\[
\mathcal{P}^\ast:=\bigcup\limits_{\substack{\mu=p_x+1,\dots,n_x \\ \nu=p_y+1,\dots,n_y}}\mathcal{P}_{\mu,\nu},
\]
may be a proper subset of $\mathcal{P}$.
\end{proof}

Note also that the results of Proposition \ref{proposition:gb_multi} and Corollary \ref{corollary:lb_multi} are independent from the type of mesh but rather rely on the partition of unity property. Therefore, a possibility is to consider local refinement strategies in order to further reduce the computational complexity and gain more flexibility only where truly needed.

\subsubsection{Shape preservation}
Shape preserving representations are crucial in geometric modeling (e.g., in CAD and CAM). Many classical quasi-interpolant strategies for function approximation preserve shape properties, such as the Bernstein approximants, the B-spline or multiquadratic (MQ) quasi-interpolants, the Variation Diminishing Spline Approximation (VDSA) and so on \cite{GAO2015696,JIANG2011,Lyche2011,Wu:1994}. 

In case of points clouds with defects, the average dataset trend is more important than the position of a single point with respect to the others. We thus introduce a notion of monotonicity and convexity for point clouds that take this consideration into account. Given a family of weight functions, we say that a point cloud is \emph{w-monotone} (resp. \emph{w-convex}) if the control point estimator $\hat{z}_w$ is monotone (resp. convex) (see Appendix \ref{Appendix} for a formal definition).

Monotonicity and convexity of the individual coordinates are preserved from $w$-monotonicity and $w$-convexity as a direct consequence of the univariate case, which is detailed in Appendix \ref{Appendix}. More precisely:
\begin{itemize}
    \item \emph{(Monotonicity)} Let us suppose that $\hat{z}_w(\cdot, y_0):\mathbb{R}\to\mathbb{R}$ is monotonic for all $y_0\in[a_2,b_2]$ (or at least it is its restriction to the nodes $\{\xi_x^{(i)}\}_{i=1}^{n_x}$). Then, $f_w$ is an increasing function of $x$ for each $y$. This statement is formally proved in Proposition \ref{proposition:uni_pres_mon}. 
    \item \emph{(Convexity)} Let us suppose that $\hat{z}_w(\cdot, y_0):\mathbb{R}\to\mathbb{R}$ is convex for all $y_0\in[a_2,b_2]$ (or at least it is its restriction to the nodes $\{\xi_x^{(i)}\}_{i=1}^{n_x}$). Then, $f_w$ is a convex function of $x$ for each $y$. This statement is formally proved in Proposition \ref{proposition:uni_pres_conv}. 
\end{itemize}

In the multivariate setting, joint monotonicity and convexity straightforwardly derive from the control net shape \cite{Goodman1989,Lyche2011}, here defined by $\hat{z}_w$. More precisely, a $w$-monotone (resp. $w$-convex) point cloud has a monotone (resp. convex) wQISA approximation.

\subsubsection{Computational complexity}
The wQISA method takes as input the point cloud $\mathcal{P} \subset \mathbb{R}^{d+1}$, the tensor product spline space $\mathbb{S}_{\mathbf{p},[\mathbf{t}_1,\ldots,\mathbf{t}_d]}$ defined by a multi-degree and a set of regular knot vectors, and the Parzen window function $w$. The approximation defined by Equation \eqref{equation:dkNNVDSA} is computed by evaluating Equation \eqref{equation:control_points_estimator} as many times as the dimension of the tensor product spline space, i.e., $\text{dim}(\mathbb{S}_{\mathbf{p},[\mathbf{t}_1,\ldots,\mathbf{t}_d]})=\prod_{i=1}^{d} n_i$.

The single control point estimation depends on the function $w$ chosen and, in particular, on its support (if global or local). In the numerical simulations proposed in Section \ref{sec:simulations}, we mainly focus on k-NN and Inverse Distance Weight (IDW) functions (see Equations \eqref{equation:kNN_weight} and \eqref{equation:Inverse_Distance_Weight}) and, therefore, we here exhibit the computational complexity of wQISA for these choices of $w$. A deepen study of the computational complexity can be found in \cite{Raffo:2020}.

The $k$-nearest neighbor can be efficiently computed using the $k$-d tree algorithm in $O(N\log(N))$ operations \cite{Friedman:1977}, where $N$ is the number of points of the cloud. The $k$-d tree then spatially stores the data in a structure such that, at runtime, the evaluation of $w$ costs $O(k)$. Thus, the computation cost of the wQISA algorithm is given by the maximum of $O(N\log(N))$ and $O(k\cdot{}\text{dim}(\mathbb{S}_{\mathbf{p},[\mathbf{t}_1,\ldots,\mathbf{t}_d]}))$.

The IDW weight is global and thus computes, for a single control point estimation, the linear combination of $N$ terms. Since computing the weight of any point is at most as expensive as the inverse of an Euclidean norm, the computational complexity is $O(Nd)$. The computational cost of the wQISA algorithm is then $O(N\cdot{}\text{dim}(\mathbb{S}_{\mathbf{p},[\mathbf{t}_1,\ldots,\mathbf{t}_d]})))$.

\section{The wQISA method from a probabilistic perspective}\label{sec:prob_interp}
Regression analysis techniques are widely used for prediction and forecasting. In regression problems, the conditional expectation of a response variable $Y$ with respect to its predictor variables $X_1,\ldots,X_p$ is often approximated by its first-order Taylor expansion. Linearity in the predictors leads to a much easier interpretability of the model and is very efficient with sparse and small data. Global and local least square approaches are among the most popular linear regression methods. Nevertheless, these models need to solve linear systems of equations, which thus unnecessarily increases computational complexity as the data size increases. 
Moreover, linear models often depend on the normal distribution of the residuals, making them unreliable when the actual distribution is asymmetric or prone to outliers.

As the assumption of linearity might be too restrictive for real-world phenomena, various methods for moving beyond it have been introduced. A popular approach, known as \emph{linear basis expansion}, considers multiple transformations of the predictors and then applies linear models in this richer space.
Compared to traditional linear models, polynomial transformations of the predictors offer a more flexible data representation as they lead to higher-order Taylor expansions. On the other hand, they suffer a lack of local shape control due to their global nature. 
Compared to polynomial bases, piecewise polynomials allow to combine an increased flexibility with a reduced number of coefficients to compute. Furthermore, nonparametric regression may be used for a variety of purposes, such as scatterplot smoothing for pure exploration and interval estimates for uncertainty examination \cite{Hastie2009}.

As we theoretically and numerically show in Sections \ref{sec:prob_interp} and \ref{sec:simulations}, the wQISA method offers a competitive alternative to handle strongly perturbed large point clouds at a reduced computational cost, even when prone to outliers.
In this Section we interpret the WQISA method as a non parametric regression problem. Independent ongoing studies on quasi-interpolation from a stochastic perspective are in \cite{Gao_2,Gao_1}.

\subsection{Formulation of the regression problem}
Let $Y$ be a univariate response variable. For the sake of simplicity, we restrict here to two predictor variables $X_1$ and $X_2$. As for the previous sections, the generalization to the multivariate case is trivial and just requires only a more involved notation.
From now on, we assume that the relationship between the predictors and the dependent variable can be expressed as the conditional expectation:
\[
\mathbb{E}(Y|X_1=x_1,X_2=x_2)=f_w(x_1,x_2).
\]
The approximation $f_w$ is here restricted to belong to a subspace of $\mathbb{S}_{\mathbf{p},[\mathbf{x}_1,\mathbf{x}_2]}$, where $\mathbf{p}\in\mathbb{N}^\ast\times\mathbb{N}^\ast$ is the (bi)-degree of the spline space over the (global) knot vectors $\mathbf{x}_1$ and $\mathbf{x}_2$. More precisely, the relation between the observations $Y_i$ and the independent variables $X_{i,1}$ and $X_{i,2}$ is formulated as
\begin{equation}
\label{equation:prob_mod}
Y_i=\sum_{j_1=1}^{n_1}\sum_{j_2=1}^{n_2}c_{j_1,j_2}B[\mathbf{x}_1^{(j_1)},\mathbf{x}_2^{(j_2)}](X_{i,1},X_{i,2})+\varepsilon_i, \quad i=1,\ldots,N,
\end{equation}
where
\begin{itemize}
\item $B[\mathbf{x}_1^{(j_1)},\mathbf{x}_2^{(j_2)}]:\mathbb{R}^2\to[0,1]$ is the $(j_1,j_2)$-th tensor product B-spline function with respect to the global knot vectors $\mathbf{x}_1$ and $\mathbf{x}_2$ respectively along $X_1$ and $X_2$.
\item $\varepsilon_i$ is the residual or disturbance term -- an unobserved random variable that perturbs the linear relationship between the dependent variable and regressors.
\end{itemize}
Relation \eqref{equation:prob_mod} can be expressed, up to a reordering of the indexes $(j_1,j_2)$, in the matrix form
\begin{equation}
\label{equation:matr_prob_mod}
\mathbf{Y}=\mathbf{B}\cdot\mathbf{c}+\boldsymbol{\varepsilon},
\end{equation}
where $\mathbf{Y}\in\mathbb{R}^{N\times1}\ni\boldsymbol{\varepsilon}$, $\mathbf{B}\in\mathbb{R}^{N\times (n_1\cdot n_2)}$ and $\mathbf{c}\in\mathbb{R}^{(n_1\cdot n_2)\times1}$. 

\subsection{Definition of the coefficient estimators}
There are different methods to fit a linear model to a given dataset. 
In the following, we introduce our new estimators for the B-spline coefficients. The $(j_1,j_2)$-th component of $\hat{\mathbf{c}}$ is defined by
\begin{equation}
\label{equation:estimators_prob_mod}
\hat{c}_{j_1,j_2}:=\dfrac{\sum\limits_{i=1}^N Y_i\cdot w_{(\xi_1^{(j_1)},\xi_2^{(j_2)})}\left(X_{i,1},X_{i,2}\right)}{\sum\limits_{i=1}^Nw_{(\xi_1^{(j_1)},\xi_2^{(j_2)})}\left(X_{i,1},X_{i,2}\right)},
\end{equation}
where $\xi_1^{(j_1)}$ and $\xi_2^{(j_2)}$  are the \emph{knot averages} with respect to the B-spline $B_{j_1,j_2}$ along the two directions. Notice that the weight functions $$w_{(\xi_1^{(j_1)},\xi_2^{(j_2)})}:\mathbb{R}^2\to[0,+\infty)$$ act both as a penalty term and as a smoother on the given data. 

\subsection{Inference for regression purposes: the bias-variance decomposition}
Suppose the data arise from a model $Y=f(X_1,X_2)+\varepsilon$. For the sake of simplicity, we assume here that the values of the predictors are fixed in advance, hence nonrandom. Further, we assume the error terms $\varepsilon_i$ to be \emph{independent identically distributed} (i.i.d) with mean $\mu_\varepsilon=0$ and variance $\sigma_\varepsilon^2$.

The generalization performances of a method relies on the simultaneously minimization of two sources of error:
\begin{itemize}
\item The \emph{bias} measures the difference between the model's expected predictions and the true values. High bias means an oversimplification of the model, i.e., the model does not produce accurate predictions (underfitting). The bias of a model is formally defined by
\begin{equation}
\label{equation:bias_defn}
\text{Bias}^2\left[\hat{f}_w(X_1,X_2)\right]:=\left(\mathbb{E}\left[\hat{f}_w(X_1,X_2)\right]-f(X_1,X_2)\right)^2.
\end{equation}
\item The \emph{variance} measures the model's sensitivity to small fluctuations in the training set. High variance can result in a model that interpolates the given data but does not generalize on data which hasn't seen before (overfitting). The variance of a model is defined by
\begin{equation}
\label{equation:var_defn}
\text{Var}\Bigl[\hat{f}_w(X_1,X_2)\Bigr]:=\mathbb{E}\left[\left(\hat{f}_w(X_1,X_2)-\mathbb{E}\left[\hat{f}_w(X_1,X_2)\right]\right)^2\right].
\end{equation}
\end{itemize}

\subsubsection{Bias of a wQISA model}
Let  $(X_1,X_2)\in[x_{1,\mu},x_{1,\mu+1})\times[x_{2,\nu},x_{2,\nu+1})$ for some $\mu=p_1+1,\ldots,n_1$ and for some $\nu=p_2+1,\ldots,n_2$. By using the property of local support of B-splines, we can then express $\mathbb{E}[\hat{f}_w(X_1,X_2)]$ as
\begin{equation}
\label{equation:Efhat}
\mathbb{E}[\hat{f}_w(X_1,X_2)]=\sum\limits_{j_1=\mu-p_1}^\mu\sum\limits_{j_2=\nu-p_2}^\nu\mathbb{E}[\hat{c}_{j_1,j_2}]\cdot B[\mathbf{x}_1^{(j_1)},\mathbf{x}_2^{(j_2)}](X_{1},X_{2}),
\end{equation}
where
\begin{equation}
\label{equation:Ecj1j2}
\begin{split}
\mathbb{E}\bigl[\hat{c}_{j_1,j_2}\bigr]
&=\dfrac{\sum\limits_{i=1}^N f(X_{i,1},X_{i,2})\cdot w_{(\xi_1^{(j_1)},\xi_2^{(j_2)})}\left(X_{i,1},X_{i,2}\right)}{\sum\limits_{i=1}^Nw_{(\xi_1^{(j_1)},\xi_2^{(j_2)})}\left(X_{i,1},X_{i,2}\right)}
\end{split}
\end{equation}
is a convex combination. Once a spline space and weight functions have been chosen, Equations \eqref{equation:Efhat} and \eqref{equation:Ecj1j2} can be combined to write down the exact formula of bias. However, in some cases bounds can simplify its study. Analogously to Proposition \ref{proposition:gb_multi}, we can compute the following bounds for $\mathbb{E}\bigl[\hat{c}_{j_1,j_2}\bigr]$
\begin{equation}
\label{equation:bounds_Ecj1j2}
\min\limits_{i\in\mathcal{I}_{\hat{c}_{j_1,j_2}}}  f(X_{i,1},X_{i,2})
\le
\mathbb{E}\bigl[\hat{c}_{j_1,j_2}\bigr]
\le
\max\limits_{i\in\mathcal{I}_{\hat{c}_{j_1,j_2}}} f(X_{i,1},X_{i,2}),
\end{equation}
where $$\mathcal{I}_{\hat{c}_{j_1,j_2}}:=\left\{i=1,\ldots,N \text{ s.t. } w_{(\xi_1^{(j_1)},\xi_2^{(j_2)})}\left(X_{i,1},X_{i,2}\right)\ne 0\right\}.$$
 
By combining Equations \eqref{equation:Efhat} and \eqref{equation:bounds_Ecj1j2}, it follows that
\begin{equation}
\label{equation:bounds_SL}
\resizebox{0.99\hsize}{!}{
\begin{math}
\begin{array}{@{}c@{\;}c@{\;}c@{\;}c@{\;}c@{}}
\min\limits_{i\in\mathcal{I}_{\mu,\nu}} f(X_{i,1},X_{i,2}) & \le & \sum\limits_{j_1=\mu-p_1}^{\mu}\sum\limits_{j_2=\nu-p_2}^{\nu}\mathbb{E}\bigl[\hat{c}_{j_1,j_2}\bigr]B[\mathbf{x}_1^{(j_1)},\mathbf{x}_2^{(j_2)}](X_1,X_2) & \le &\max \limits_{i\in\mathcal{I}_{\mu,\nu}}  f(X_{i,1},X_{i,2}) \\
\vdequal &&  \vdequal  && \vdequal  \\
\alpha(\mu,\nu) && \mathbb{E}\bigl[\hat{f}_w(X_1,X_2)\bigr] && \beta(\mu,\nu) \\
\end{array}
\end{math}
}
\end{equation}
where $$\mathcal{I}_{\mu,\nu}:=\bigcup\limits_{\substack{j_1=\mu-p_1,\ldots,\mu \\ j_2=\nu-p_2,\ldots,\nu}}\mathcal{I}_{\hat{c}_{j_1,j_2}}$$  and where $\alpha$ and $\beta$ denote the lower and upper bounds. We conclude that

\begin{equation}
\label{equation:bounds_bias}
\text{Bias}^2\left[\hat{f}_w(X_1,X_2)\right] 
\begin{cases}
\le  \bigl(\alpha(\mu,\nu) - f(X_1,X_2)\bigr)^2, & \text{if } \mathbb{E}\bigl[\hat{f}_w(X_1,X_2)\bigr]\le f(X_1,X_2) \\

\le
\bigl(\beta(\mu,\nu) - f(X_1,X_2)\bigr)^2, & \text{if } \mathbb{E}\bigl[\hat{f}_w(X_1,X_2)\bigr]\ge f(X_1,X_2) \\
\end{cases},
\end{equation}
where $\alpha(\mu,\nu)$ and $\beta(\mu,\nu)$ denote the minimum of maximum in  Equation \eqref{equation:bounds_SL}.

\subsubsection{Variance of a wQISA model}
In the following Lemma we provide an exact formula for the variance while proving that, in the worst possible case, the variance will still be upper bounded by $\sigma_\varepsilon^2$.  For the sake of simplicity, we consider a reordering of B-splines as in Equation \eqref{equation:matr_prob_mod}. This choice allows to substitute the indexes $(j_1,j_2)$ with a single index $j$.

\begin{lemma} 
\label{lemma:variance_estimator}
The variance of $\hat{f}_w$ is upper-bounded by the variance of the error, i.e.,
$$\text{Var}\Bigl[\hat{f}_w(X_1,X_2)\Bigr]\le\sigma^2_\varepsilon.$$ 
\end{lemma}
\begin{proof}
\begin{equation*}
\begin{split}
\text{Var}\Bigl[\hat{f}_w(X_1,X_2)\Bigr]& =\mathbb{E}\left[\left(\hat{f}_w\left(X_1,X_2\right)-\mathbb{E}\bigl(\hat{f}_w\left(X_1,X_2\right)\bigr)\right)^2\right]= \\
&=\mathbb{E}\left[\left(\sum_i\hat{c_i}B_i(X_1,X_2)-\sum_i\mathbb{E}\left[\hat{c_i}\right]B_i(X_1,X_2)\right)^2\right]= \\
&=\mathbb{E}\left[\left(\sum_i\left(\hat{c_i}-\mathbb{E}\left[\hat{c_i}\right]\right)B_i(X_1,X_2)\right)^2\right]= \\
&=\sum_i\sum_j\mathbb{E}\left[\left(\hat{c_i}-\mathbb{E}\left[\hat{c_i}\right]\right)\left(\hat{c_j}-\mathbb{E}\left[\hat{c_j}\right]\right)\right]B_i(X_1,X_2)B_j(X_1,X_2)= \\
&=\sum_i\sum_j\text{Cov}\left(\hat{c_i},\hat{c_j}\right)B_i(X_1,X_2)B_j(X_1,X_2),
\end{split}
\end{equation*}
where
\begin{equation*}
\resizebox{0.99\hsize}{!}{
\begin{math}
\begin{aligned}
\text{Cov}\left(\hat{c_i},\hat{c_j}\right)&=\text{Cov}\left(\dfrac{\sum_{k_1}Y_{k_1}\cdot w_{(\xi_1^{(i)},\xi_2^{(i)})}\left(X_{k_1,1},X_{k_1,2}\right)}{\sum_{k_1} w_{(\xi_1^{(i)},\xi_2^{(i)})}\left(X_{k_1,1},X_{k_1,2}\right)},\dfrac{\sum_{k_2}Y_{k_2}\cdot w_{(\xi_1^{(j)},\xi_2^{(j)})}\left(X_{k_2,1},X_{k_2,2}\right)}{\sum_{k_2} w_{(\xi_1^{(j)},\xi_2^{(j)})}\left(X_{k_2,1},X_{k_2,2}\right)}\right)=\\
&=\dfrac{\sum_{k_1}\sum_{k_2}w_{(\xi_1^{(i)},\xi_2^{(i)})}\left(X_{k_1,1},X_{k_1,2}\right)\cdot{}w_{(\xi_1^{(j)},\xi_2^{(j)})}\left(X_{k_2,1},X_{k_2,2}\right)\text{Cov}\left({Y_{k_1}},{Y_{k_2}}\right)}{\sum_{k_1}\sum_{k_2}w_{(\xi_1^{(i)},\xi_2^{(i)})}\left(X_{k_1,1},X_{k_1,2}\right)\cdot{}w_{(\xi_1^{(j)},\xi_2^{(j)})}\left(X_{k_2,1},X_{k_2,2}\right)}=\\
&=\sigma^2_\varepsilon\dfrac{\sum_{k_1}w_{(\xi_1^{(i)},\xi_2^{(i)})}\left(X_{k_1,1},X_{k_1,2}\right)\cdot{}w_{(\xi_1^{(j)},\xi_2^{(j)})}\left(X_{k_1,1},X_{k_1,2}\right)}{\sum_{k_1}\sum_{k_2}w_{(\xi_1^{(i)},\xi_2^{(i)})}\left(X_{k_1,1},X_{k_1,2}\right)\cdot{}w_{(\xi_1^{(j)},\xi_2^{(j)})}\left(X_{k_2,1},X_{k_2,2}\right)}.
\end{aligned}
\end{math}
}
\end{equation*}
Thus
\begin{equation*}
\resizebox{0.99\hsize}{!}{
\begin{math}
\begin{aligned}
\text{Var}\Bigl[\hat{f}_w(X_1,X_2)\Bigr]&=
\sum_i\sum_j\text{Cov}\left(\hat{c_i},\hat{c_j}\right)B_i(X_1,X_2)B_j(X_1,X_2)=
\\
&=
\sigma^2_\varepsilon\sum_i\sum_j\dfrac{\sum_{k_1}w_{(\xi_1^{(i)},\xi_2^{(i)})}\left(X_{k_1,1},X_{k_1,2}\right)\cdot{}w_{(\xi_1^{(j)},\xi_2^{(j)})}\left(X_{k_1,1},X_{k_1,2}\right)}{\sum_{k_1}\sum_{k_2}w_{(\xi_1^{(i)},\xi_2^{(i)})}\left(X_{k_1,1},X_{k_1,2}\right)\cdot{}w_{(\xi_1^{(j)},\xi_2^{(j)})}\left(X_{k_2,1},X_{k_2,2}\right)}\cdot \\
&\cdot B_i(X_1,X_2)B_j(X_1,X_2)\le\sigma^2_\varepsilon,
\end{aligned}
\end{math}
}
\end{equation*}
where the inequality holds because $B_iB_j$ has the partition of unity property.
\end{proof}

In Lemma \ref{lemma:variance_estimator}, the exact expression of the variance makes it possible to  compute exact and approximated (pointwise) standard error bands (see Equation \eqref{equation:se_band}).

\subsubsection{Bias-variance decomposition for the $k$-NN weight}
Let's consider an example to show how the results of the section work in practice. Let $w$ be a $k$-NN weight function (see Equation \eqref{equation:kNN_weight}). The exact expression of the bias is found by combining Equation \eqref{equation:bias_defn} with the expected value
\begin{equation*}
    \begin{split}
    \mathbb{E}[\hat{f}_w(X_1,X_2)]&=\dfrac{1}{k}
    \sum\limits_{j_1=\mu-p_1}^{\mu}\sum\limits_{j_2=\nu-p_2}^{\nu}
    B[\mathbf{x}_1^{(j_1)},\mathbf{x}_2^{(j_2)}](X_1,X_2)\cdot{} \\
    &\cdot{}\sum\limits_{(X_{i,1},X_{i,2})\in N_k(\xi_1^{(j_1)},\xi_2^{(j_2)})}f(X_{i,1},X_{i,2}).
    \end{split}
\end{equation*}
The exact expression of variance is given by
\begin{equation*}
    \text{Var}\Bigl[\hat{f}_w(X_1,X_2)\Bigr]=\sigma^2_\varepsilon\sum_i\sum_j\dfrac{k_{i,j}}{k^2}\cdot B_i(X_1,X_2)B_j(X_1,X_2)\le\dfrac{\sigma^2_\varepsilon}{k},
\end{equation*}
where $k_{i,j}$ is the number of points in common, if any, among the $k$-closest to $(\xi_1^{(i)},\xi_2^{(i)})$ and $(\xi_1^{(j)},\xi_2^{(j)})$.

For small $k$, the estimate $\hat{f}_w$ can potentially adapt itself better to the underlying $f$, as it will avoid points further away to the knot averages. Under the assumption of increasing the point cloud size while keeping the sampling uniform, the bias for the 1-NN weight function vanishes entirely as the size of the training set approaches infinity and the mesh is uniformly refined. On the other hand, larger values of $k$ can decrease the variance.

\subsubsection{Numerical interpretation of the bias-variance decomposition}

Figure \ref{figure:bias_variance} shows the effect of spline spaces of different dimensions on the simple example
\begin{equation*}
    Y=\sin{\pi X}+\varepsilon,
\end{equation*}
with $X\sim U[-2,2]$ and $\varepsilon\sim N(0,\sigma^2)$. Our dataset consists of $N=300$ points $(x_i,y_i)$ sampled on the exact curve and then perturbed.

The weighted quasi interpolant spline approximations for three different uniform knot vectors are shown. For the sake of simplicity, we here considered a $10$-NN weight function. 
The shaded regions in the figures represent the (pointwise) standard error band of $\hat{f}_w$, i.e., the region 
\begin{equation}
    \label{equation:se_band}
    \hat{f}_w(X)\pm z^{(1-\alpha)}\cdot\sqrt{Var\left[\hat{f}_w(X)\right]},
\end{equation}
where $z^{(1-\alpha)}$ is the $1-\alpha$ percentile of the normal distribution.
The three approximations displayed in Figures \ref{figure:bias_variance}(b-d) give a graphical representation of the bias-variance trade-off problem with respect to the dimension of the spline space:
\begin{itemize}
    \item[\textbf{n=5}]  The spline under-fits the data, with a more dramatic bias in those regions with a higher curvature
    \item[\textbf{n=15}] Compared to the previous case, the fitted function is closer to the true function. The variance has not increased appreciably yet.
    \item[\textbf{n=30}] The spline over-fits the data, which leads to a locally increased width of the bands.
\end{itemize}

In practice, the tuning parameters (here: $n$) can be selected via automatic procedures, for instance by using the $K$-fold cross-validation, generalized cross-validation and the so-called $C_p$ statistic \cite{Hastie2009}. In Figure \ref{figure:bias_variance}(a) we include the $5$-fold cross-validation curve
\begin{equation*}
    CV(n)=\dfrac{1}{N}\sum_{i=1}^N\left(y_i-\hat{f}_w(x_i)\right)^2,
\end{equation*}
where $f_w$ depends on $n$ via the spline space.

\begin{figure}[!ht]
\begin{tabular}{cc}
\includegraphics[scale=0.4]{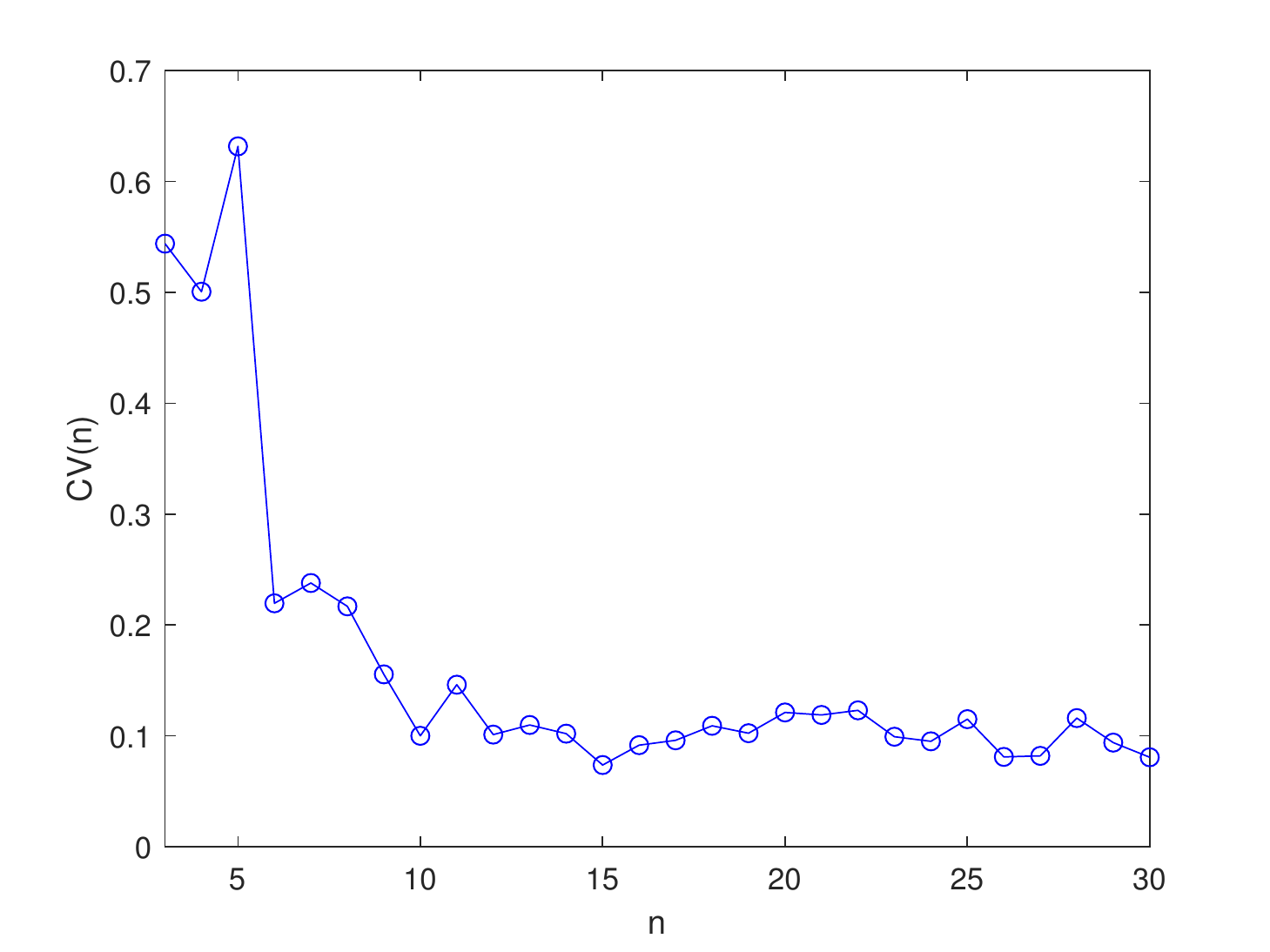}
&
\includegraphics[scale=0.4]{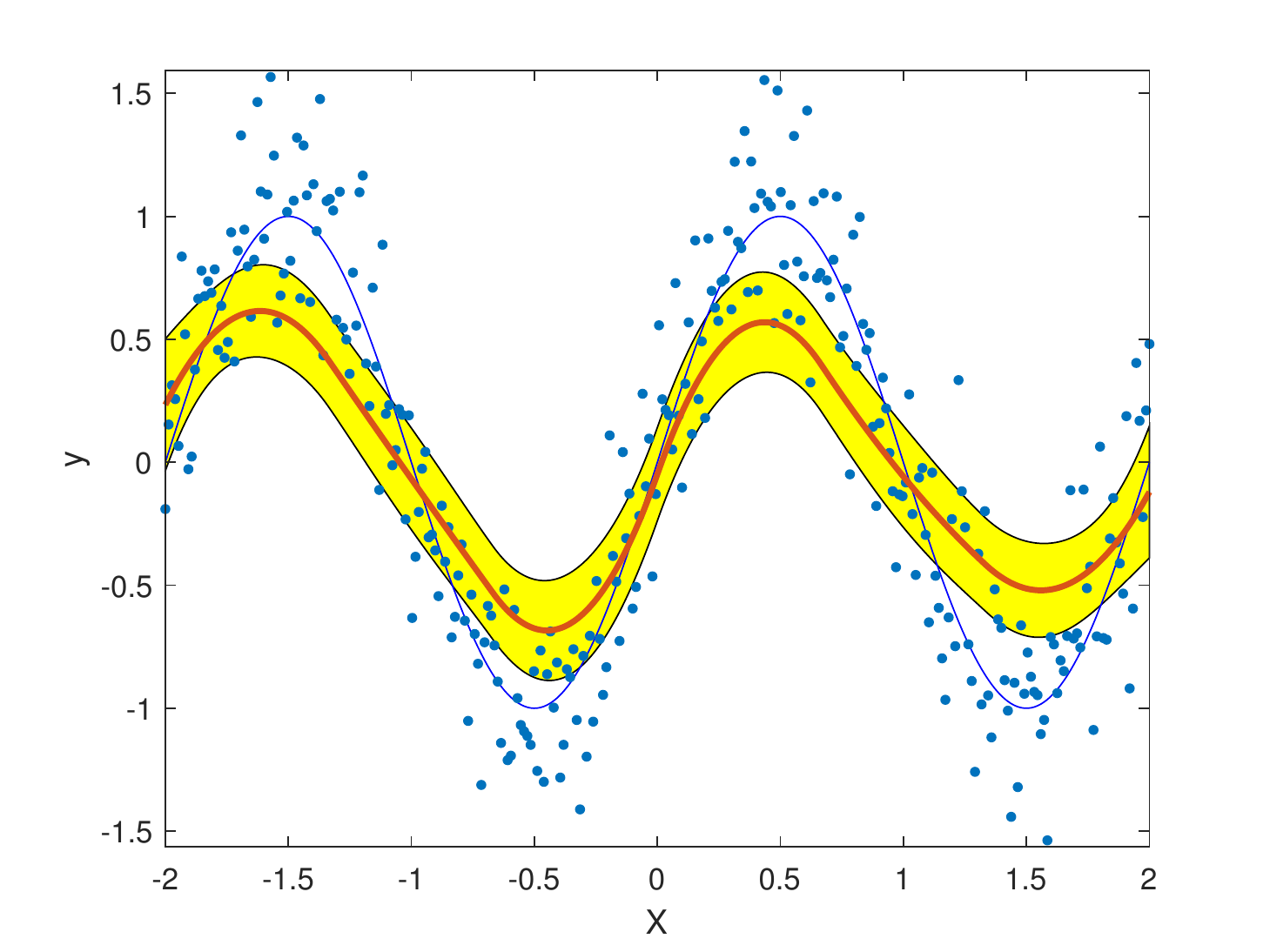}
\\
(a) & (b) \\
\includegraphics[scale=0.4]{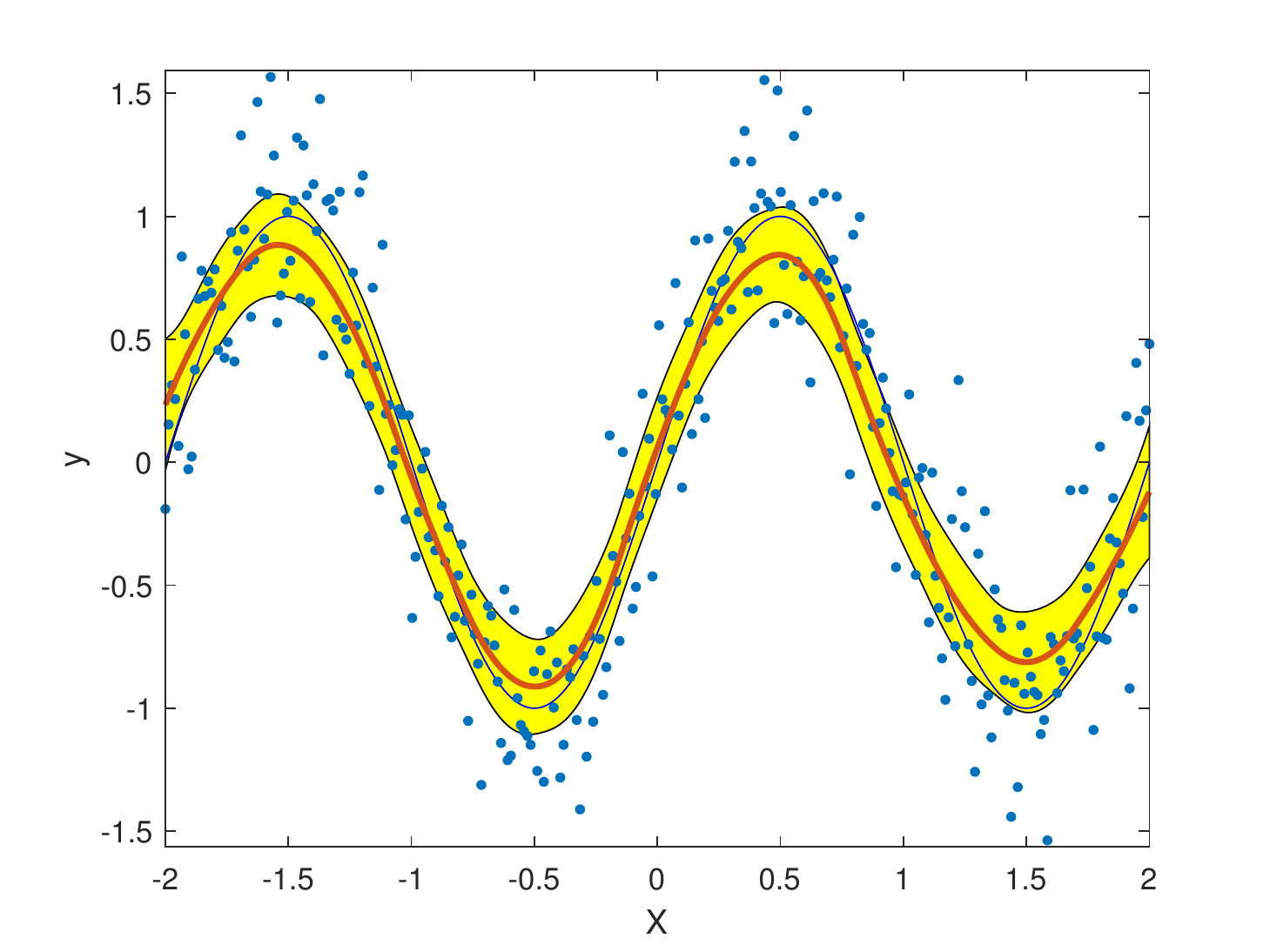}
&
\includegraphics[scale=0.4]{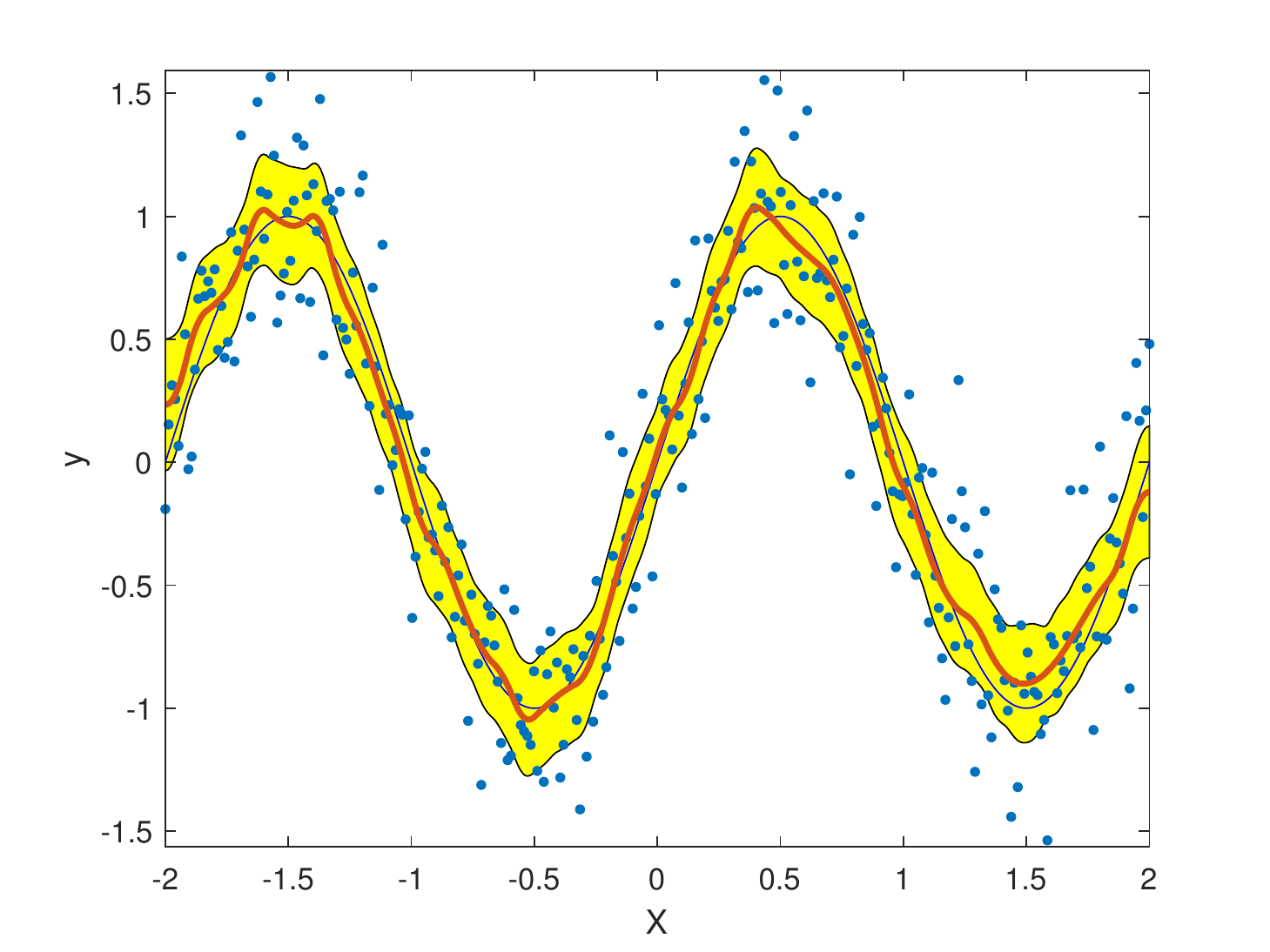}
\\
(c) & (d)
\end{tabular}
\caption{Bias-variance tradeoff. In (a) we show the CV(n) curve  for a realization from the chosen nonlinear additive error model. The minimum is reached at $n=15$. The remaining panels show the data, the true function (in blue), the weighted quasi interpolant spline approximations (in red) and the (yellow shaded)  bands of Equation \eqref{equation:se_band}, for spline spaces of dimension $n=5$ (b), $n=15$ (c) and $n=50$ (d). The bands corresponds here to an approximate $95\%$ confidence interval.}
\label{figure:bias_variance}
\end{figure}

Figure \ref{figure:variable_noise} shows the approximation of a point cloud affected by the non-uniform noise $\varepsilon \sim N(0,s(X))$, defined as follows:
\begin{equation*}
    s(X)=e^{-\dfrac{1}{4(1+e^{4X-2})}}.
\end{equation*}
The dataset consists of $N=400$ points and is approximated by a spline space containing $n=15$ B-splines over a uniform knot vector. 

\begin{figure}[!ht]
\centering
\includegraphics[scale=0.4]{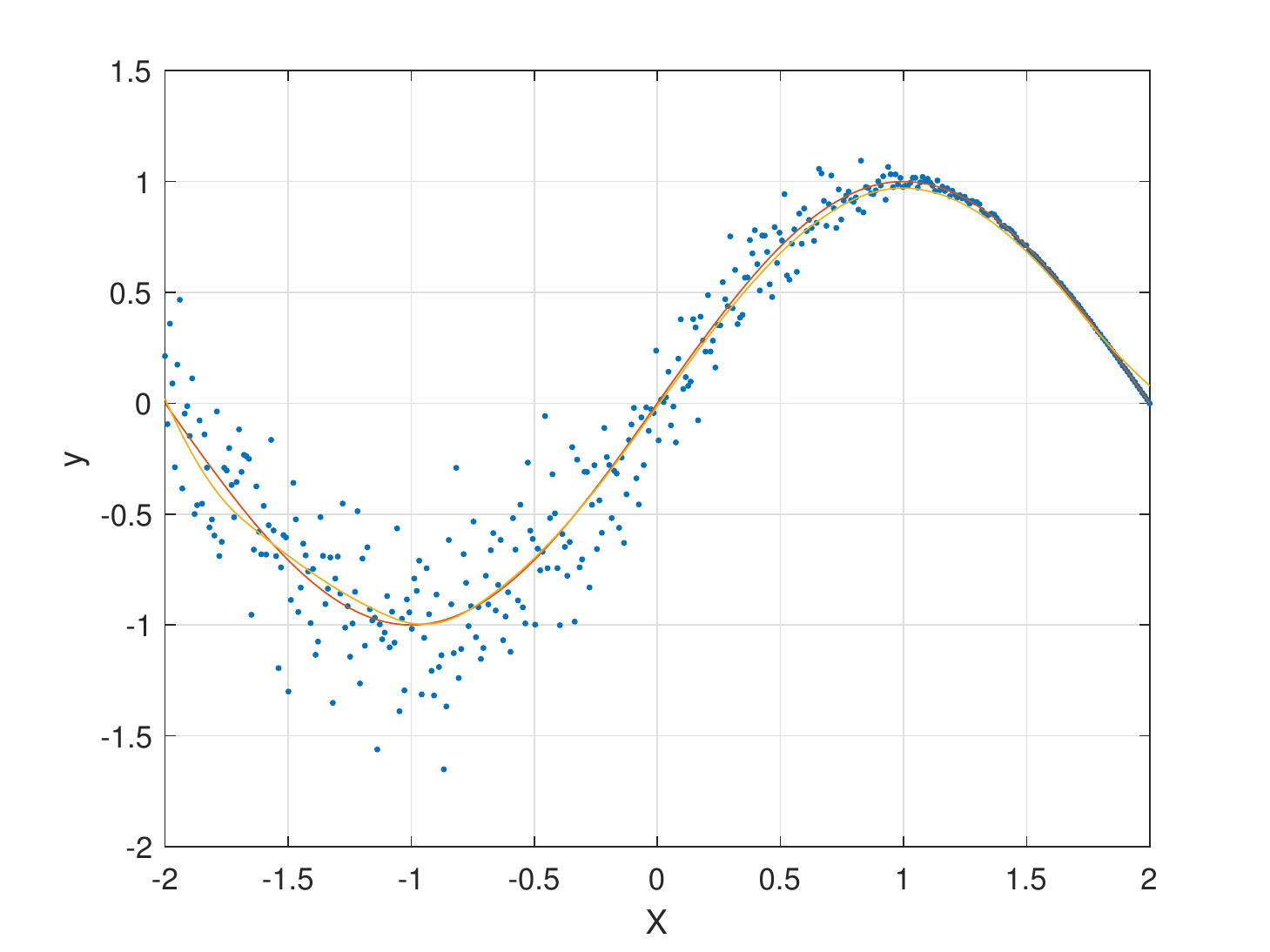}
\caption{Variable noise approximation. We show the original curve $f(X)=\sin{(\pi/2\cdot X)}$ (in red) and the spline approximation (in yellow).}
\label{figure:variable_noise}
\end{figure}

\newpage\phantom{blabla}
\newpage\phantom{blabla}

\section{Numerical simulations}
\label{sec:simulations}
We draw the effectiveness of our method in a number of real data coming from different sources and application domains. While \cite{Raffo:2020} focused on the local approximation of 3D point clouds by wQISA surfaces, this section shows how the method is able to address the approximation problem for different dimensions. Indeed, our examples include curve approximation (on images and 3D objects), surface approximation (of 3D point clouds) and simulation of natural phenomena (like rainfall precipitation) over surfaces. Unless otherwise stated, we focus here on (bi-)quadratic spline approximations defined over uniform knot vectors, as they provide a sufficient flexibility for our purposes. Nevertheless, one can consider \mbox{(bi-)degrees} as additional parameters to assess and perform knot insertion to increase the degrees of freedom only where they are actually needed.

\subsection{Evaluation criteria}
\label{sec:error}
The data acquisition devices and the subsequent post-processing operations generally introduce geometric and numerical artefacts. Unfortunately, for most of the data, the information on the quality of the acquisition devices and type of post-processing operations are lost or not available. Therefore, the hypothesis that the data to be approximated are \emph{exact} is often unrealistic. Differently from other model representations, the peculiarity of wQISA is its capability of dealing with data affected by noise and outliers. This fact reflects on the measurements we can adopt to analyse the quality of the data approximation: indeed, it is not important how much the wQISA interpolates the original data rather it remains in a reasonable approximation range. To the best of our knowledge, a single performance measure able to capture such complex information does not exist; therefore, we will analyse the wQISA output with a number of measures, each one able to highlight different approximation aspects.
\begin{itemize}
    \item When $N$ observations $Y_i$ are approximated by $\hat{Y}_i$, two popular measures of the statistical dispersion are the \emph{Mean Squared Error} (MSE) $$MSE=\dfrac{1}{N}\sum_{i=1}^N (Y_i -\hat{Y}_i)^2$$ and the \emph{Mean Absolute Error} (MAE) $$MAE=\dfrac{1}{N}\sum_{i=1}^N |Y_i -\hat{Y}_i|.$$ Although the MSE and MAE quantities are sample dependent and highly affected by data perturbation, they offer a very intuitive quantification of how close a point cloud and its approximation are.
    \item The \emph{Hausdorff distance} is a well-known distance between two sets of points and applies for point clouds in all dimensions. In particular, we consider the Directed Hausdorff distance \cite{DezaDeza} from the points $a\in A\subset\mathbb{R}^t$ to the points $b\in B\subset\mathbb{R}^t$ as follows: 
	$$d_{dHaus}(A,B)=\max\limits_{a\in A}\min\limits_{b\in B}d(a,b),$$ with $d$ the Euclidean distance. In order to have a coherent distance evaluation through models of different size, we normalize $d_{dHaus}$ with respect to the diameter of the point cloud.

    \item The \emph{Jaccard index} (also known as intersection over union) quantitatively estimates how two sets overlap. It is has been previously adopted to measure the performance of curve recognition methods for images \cite{Taha2015MetricsFE} and 3D models \cite{shrec19}. The Jaccard index between two point sets $A$ and $B$ is defined as: $$Jaccard(A,B)=\dfrac{|A \cap B|}{|A \cup B|},$$ where $|\cdot|$ denotes here the number of elements. The Jaccard index varies from 0 to 1, the higher the better. In our context, it can be adapted to the ratio of elements of the original point cloud that lie on the standard error bands of Equation \eqref{equation:se_band}.
\end{itemize}

\subsection{Curve approximation}
We consider a $512\times512$ axial X-ray CT slice of a human lumbar vertebra (see Figure \ref{figure:XrayCT}(a)). First, we apply an edge detection technique, to detect the set of edge points. In this specific example, we adopt the Canny edge detector \cite{Canny:1986}, others methods could be applied too. We select a bounding box for the point cloud, which is then partitioned in smaller sub-regions (see Figure \ref{figure:XrayCT}(b)). Lastly, we apply our technique to each sub-region to obtain a global approximation  (see Figure \ref{figure:XrayCT}, right). Here, a $1$-NN weight is set as the number of points is relatively small. Uniform knot vectors are considered as they produce reasonable approximations. The number of B-splines is chosen, in each sub-region, by a Leave-One-Out cross-validation \cite{Hastie2009}.  Interpolating conditions are imposed at the boundaries in order to have a more natural $C^1$ continuity (see Figure \ref{figure:XrayCT}(c)). Notice that the shape of the vertebra is correctly preserved in the passage from the image to the final approximation.

\newcolumntype{S}{>{\centering\arraybackslash} m{1.3in} }
\begin{figure}[htp]
\begin{center}
\begin{tabular}{S S S}
\includegraphics[scale=0.34,trim={2.5cm 0 2cm 0},clip]{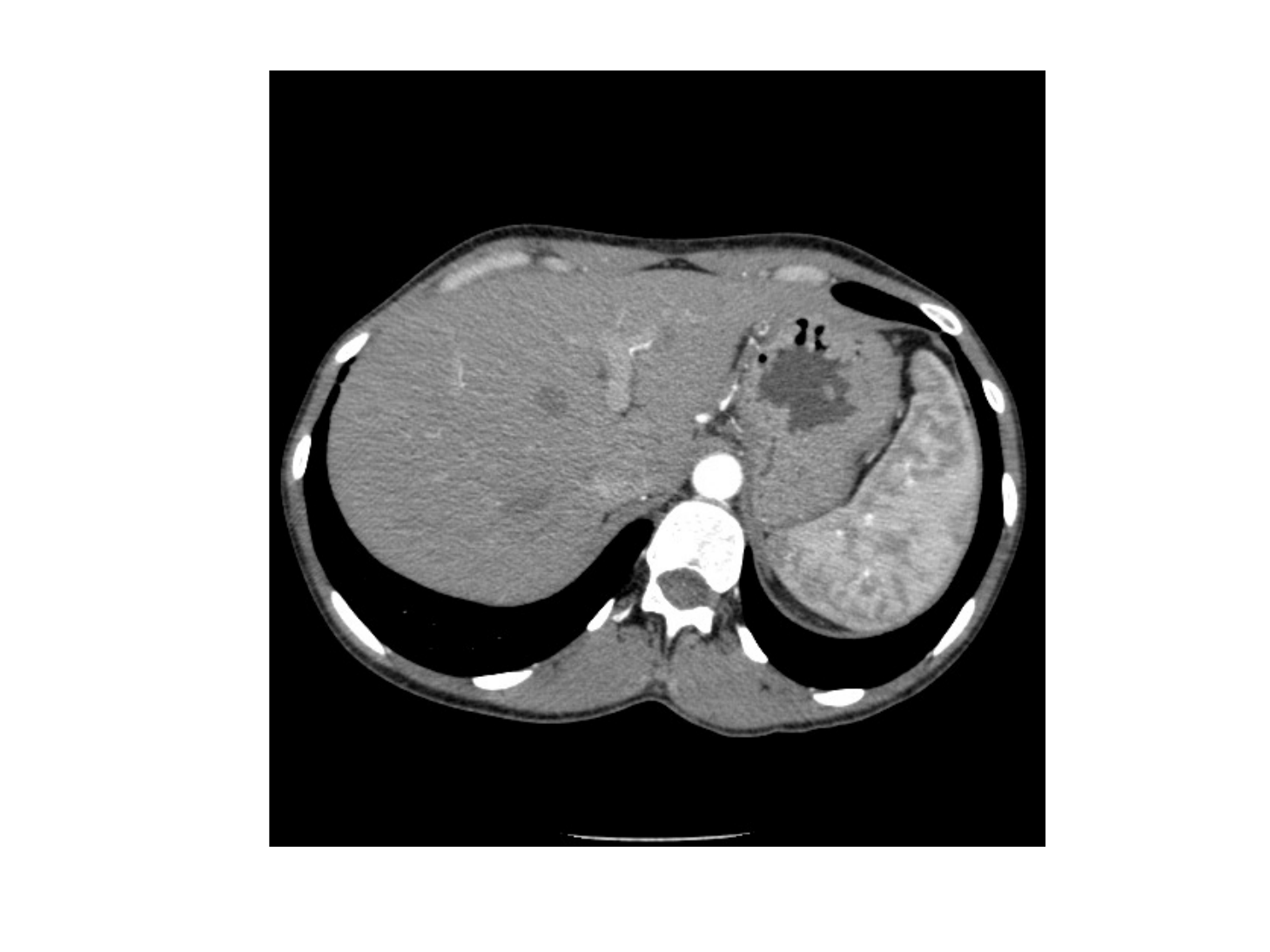}
&
\includegraphics[scale=0.37,trim={3cm 0 3cm 0},clip]{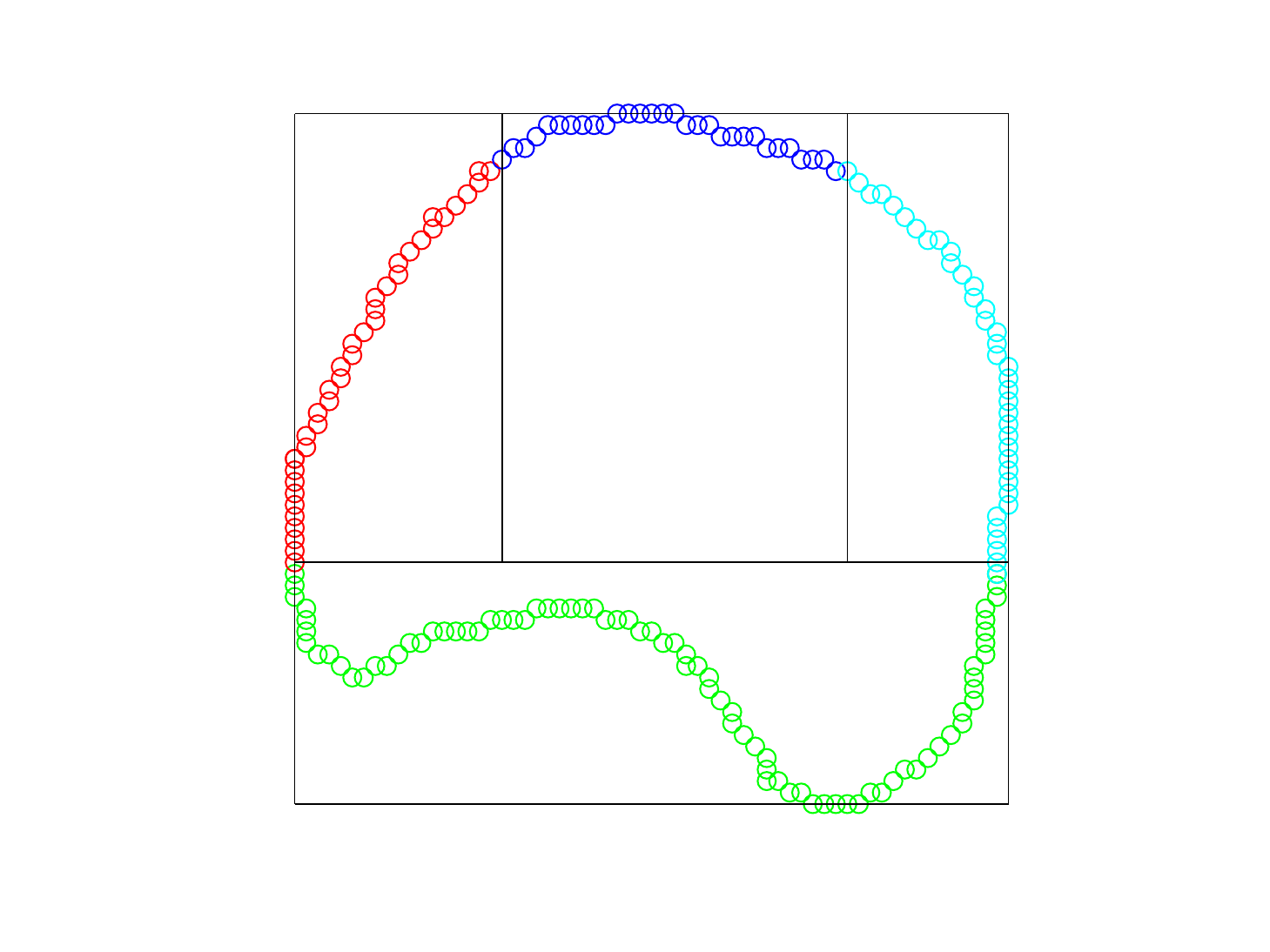}
&
\includegraphics[scale=0.34,trim={2cm 0 2cm 0},clip]{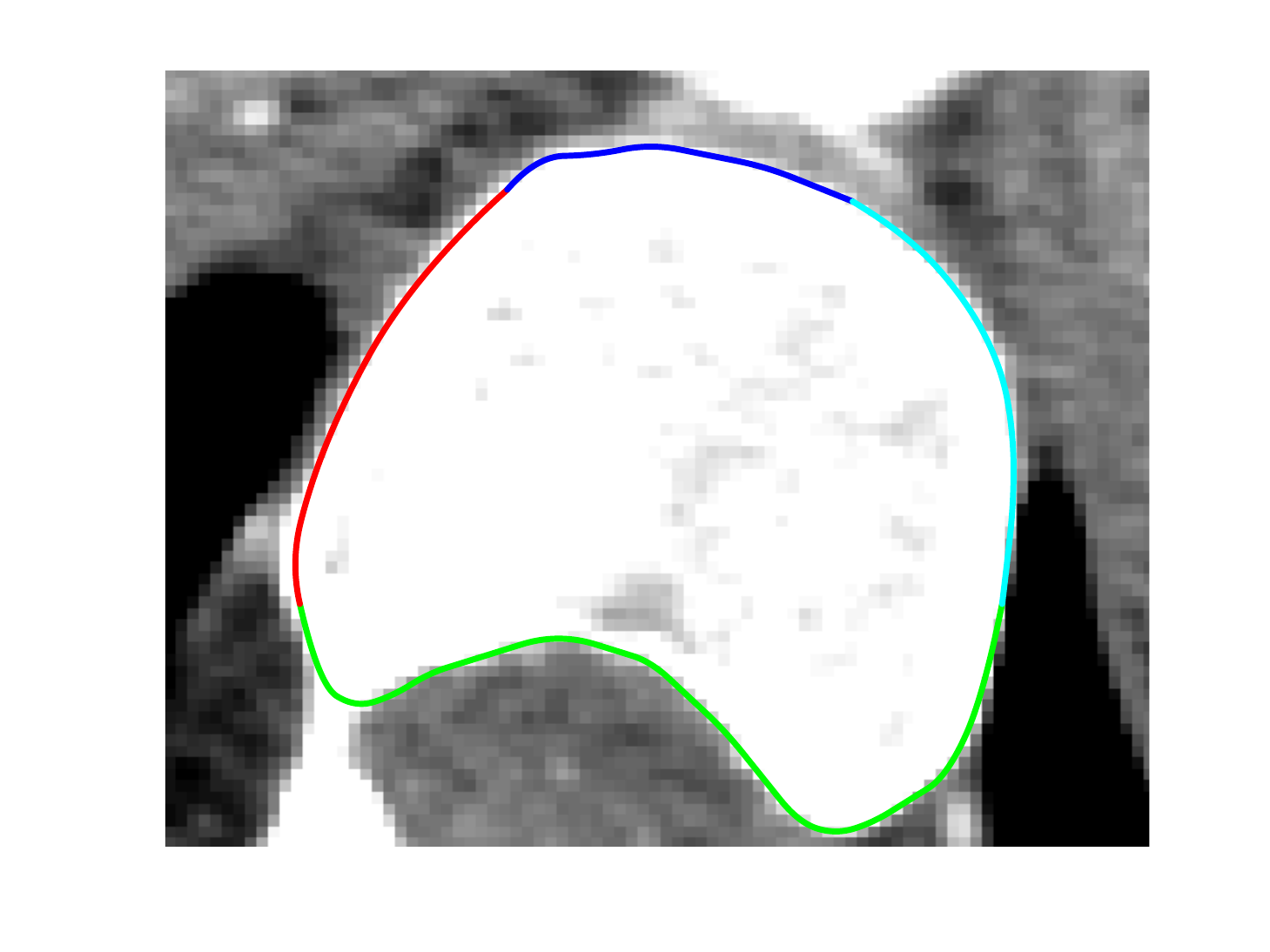} \\
(a) & (b) & (c) \\
\end{tabular}
\end{center}
\caption{X-ray CT slice. In (a), the original image is shown. Figure (b) displays the edge points and the chosen partition: V1 in green, V2 in red, V3 in blue and V4 in light blue. In (c), the piecewise defined curve is superimposed to an enlargement of the original image.}
\label{figure:XrayCT}
\end{figure}

Figure \ref{figure:eye} shows an example of eye contour approximation from 3D models. We consider a fragment of a votive statue \cite{vassos1991} stored in STARC repository\footnote{http://public.cyi.ac.cy/starcRepo/} at The Cyprus Institute and extract the eye contours by filtering the point cloud through the values of the mean curvature values and clustering \cite{Torrente2018}. Each contour is projected onto its regression plane and then locally approximated. We here test a $k$-NN weight, with $k$ to be assessed from patch to patch. Knot vectors are again assumed to be uniform. For each eye profile, two curves are detected; the extrema knots of the two curves are fixed to be the same and are automatically selected as the leftmost and rightmost points of the whole profile. Notice that with these choices our approach is also able to fill the gaps in a reasonable way.

\begin{figure}[!h]
\begin{center}
\begin{tabular}{cc}
\includegraphics[scale=1.2,trim={5cm 0 5cm 0}]{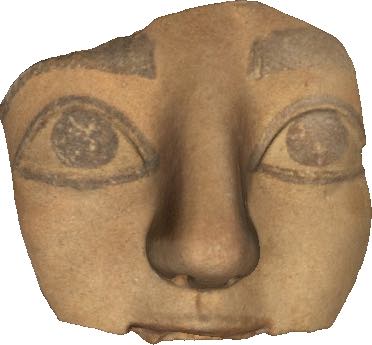}
&
\includegraphics[scale=1.2,trim={0cm 0 0cm 0}]{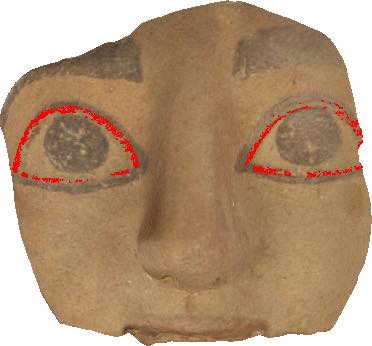} \\
(a) & (b) \\
\includegraphics[scale=0.35,trim={0cm 0 0cm 0}]{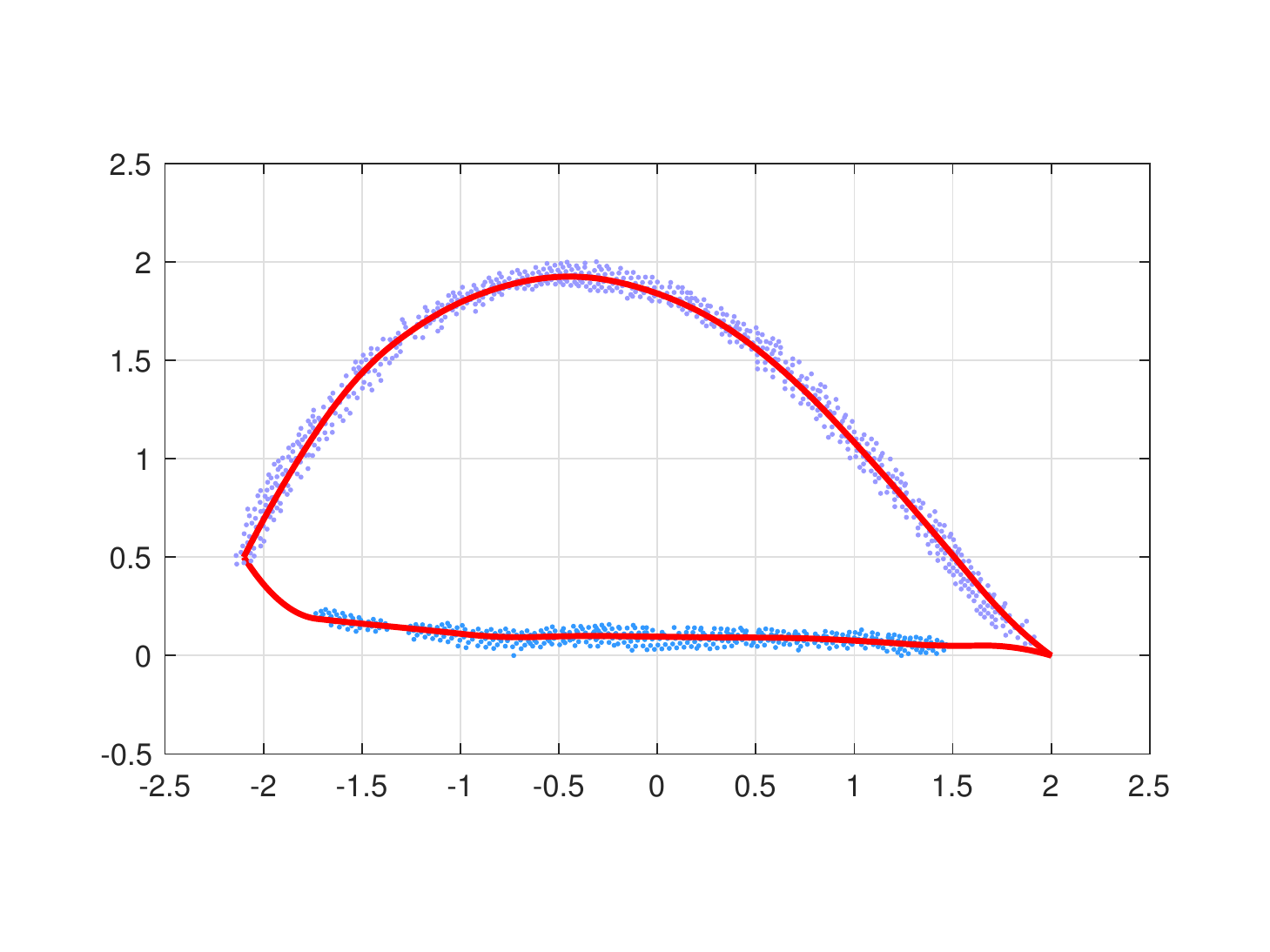}
&
\includegraphics[scale=0.35,trim={0cm 0 0cm 0}]{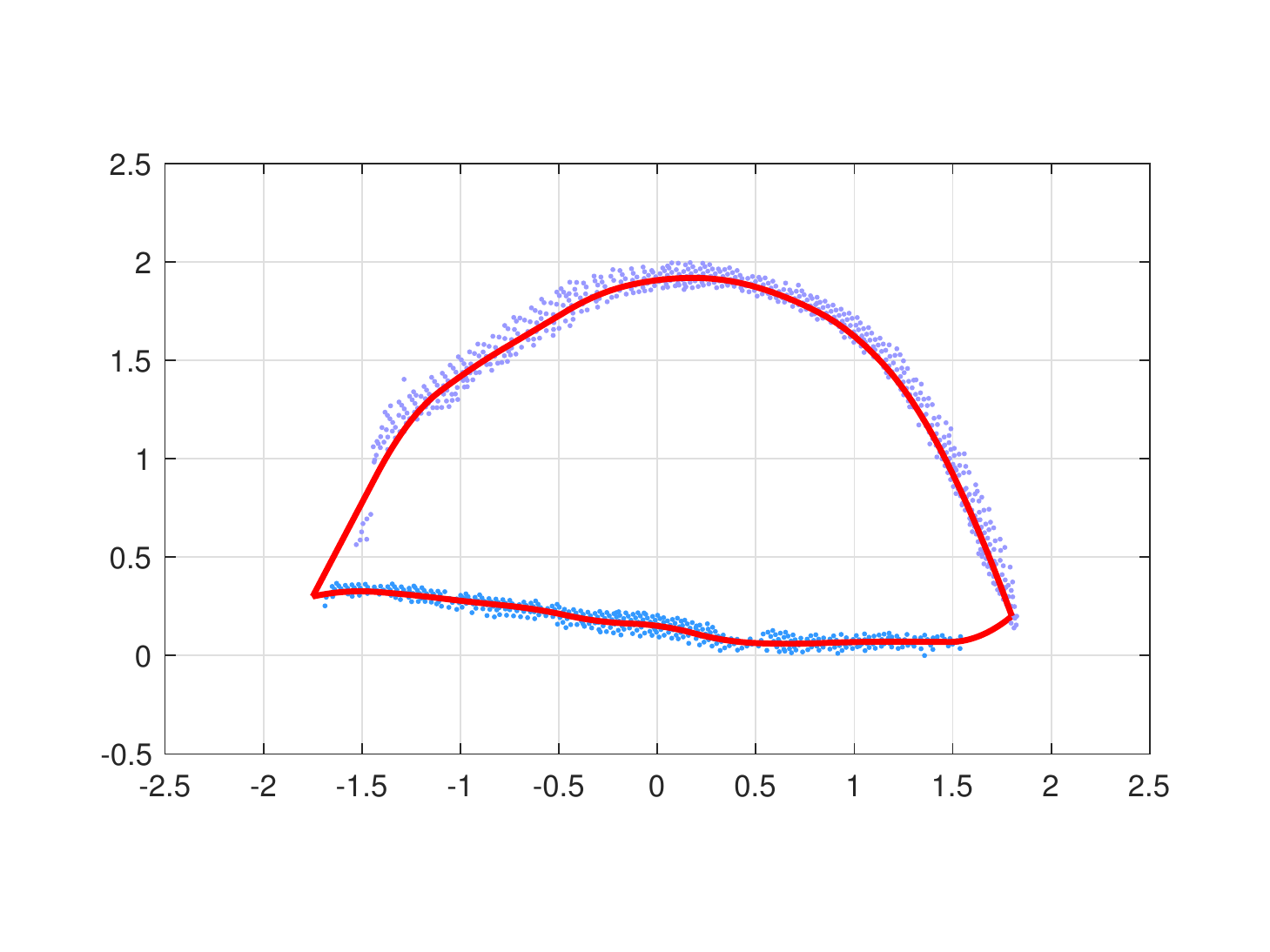} \\
(c) & (d)
\end{tabular}
\end{center}
\caption{Approximation of the eye contour on a fragment of archaeological artifact. The statue (a) is first preprocessed to filter the eye contours points (b). Then, each point cloud is locally approximated. Here the points are clustered into: LE1 (left eye, light blue), LE2 (left eye, light purple), RE1 (right eye, light blue), RE2 (right eye, light purple).}
\label{figure:eye}
\end{figure}

Table \ref{tab:curve_approximations} reports the values of the parameters $n$ and $k$  that best approximate the original curve segments and the corresponding error measures for the wQISA approximations.

\begin{table}[h!]
\centering
\caption{Parameters and accuracy measures for the curve fitting examples. For each cluster of points we report: the sample size, the number of B-splines $n$, the tuning parameter $k$ for the $k$-NN weight, the Mean Absolute Error (MAE), the Root Mean Squared Error (RMSE), the Jaccard index and the normalized Hausdorff distance. Parameters with asterisks are set by user.} 
\begin{tabular}{c c c c c c c c c}
& \multicolumn{4}{c}{\textbf{Lumbar Vertebra}}
& \multicolumn{2}{c}{\textbf{Left Eye}}
& \multicolumn{2}{c}{\textbf{Right Eye}} \\ 
\cmidrule(l){2-9}
  & V1 & V2 & V3 & V4 & LE1 & LE2 & RE1 & RE2\\ 
\midrule 
sample size & 82 & 38 & 30 & 38 & 422 & 730 & 428 & 638 \\ 
$n$ & 20 & 6 & 8 & 7 & 12 & 12 & 8 & 12 \\ 
$k$ & $1^\ast$ & $1^\ast$ & $1^\ast$ & $1^\ast$ & 5 & 5 & 5 & 5\\ 
\midrule
MAE & 0.656 & 0.3323 & 0.278 & 0.292 & 0.025 & 0.052 & 0.025 & 0.043 \\
RMSE & 0.898 & 0.418 & 0.378 & 0.379 & 0.030 & 0.074 & 0.030 & 0.079 \\
Jaccard & 0.988 & 1.000 & 1.000 & 1.000 & 0.995 & 1.000 & 1.000 & 1.000 \\ 
Hausdorff & 0.014 & 0.016 & 0.011 & 0.012 & 0.010 & 0.021 & 0.017 & 0.041\\ 
\midrule
\end{tabular}
\label{tab:curve_approximations}
\end{table}

\subsection{Surface approximation}
A simulation on terrain data is shown in Figure \ref{figure:portofino}. The data are part of the Liguria-LAS dataset adopted as testbed in the iQmulus project \cite{iQmulus}, and come from a LIDAR dataset with spatial resolution of one meter. 
The area here selected contains 379.831 points. It is located in the Liguria region, in the north-west of Italy. The Liguria morphology, with several small catchments and even small rivers, is very challenging for the approximation methods to capture and preserve the most important and potentially critical characteristics \cite{PATANE_2017}.
The data are obtained with multiple swipes by airplane lidar acquisition. Some points come from multiple laser positions and therefore the same point can have multiple elevation values.
In addition, since the data were only minimally post-processed to convert them to .las format, they contain also noise and outliers. 
In this example, we choose a $C^1$ (bi-)quadratic spline approximation because it is smooth enough to represent smooth terrains in a good way. We consider an Inverse Distance Weight (IDW), defined by:
\begin{equation}
\label{equation:Inverse_Distance_Weight}
    w_{(u,v)}(x,y):=
    \begin{cases}
    \dfrac{1}{||(x,y)-(u,v)||_2}, & \text{ if } |C_{(u,v)}|=0 \\
    \begin{cases}
    \dfrac{1}{|C_{(u,v)}|}, &   \text{ for all } (x,y)=(u,v) \\
    0,                      &   \text{ else }    
    \end{cases}
    , & \text{ if } |C_{(u,v)}|\ne 0
    \end{cases}
\end{equation}
where $|C_{(u,v)}|:=\{(x,y,z)\in\mathcal{P} \text{ s.t. } (x,y)=(u,v)\}$.
The IDW assigns greater influence to the points the closest to the knot averages and hence the most significant for the terrain approximation. The uniform knot vectors define in the final approximation $1024$ B-splines in both directions and are chosen such that the MSE for the relative punctual error of each element is lower than $0.05$ (which correspond to 0.05\% of deviation).

\begin{figure}[!ht]
\begin{center}
\begin{tabular}{cc}
\includegraphics[scale=0.45]{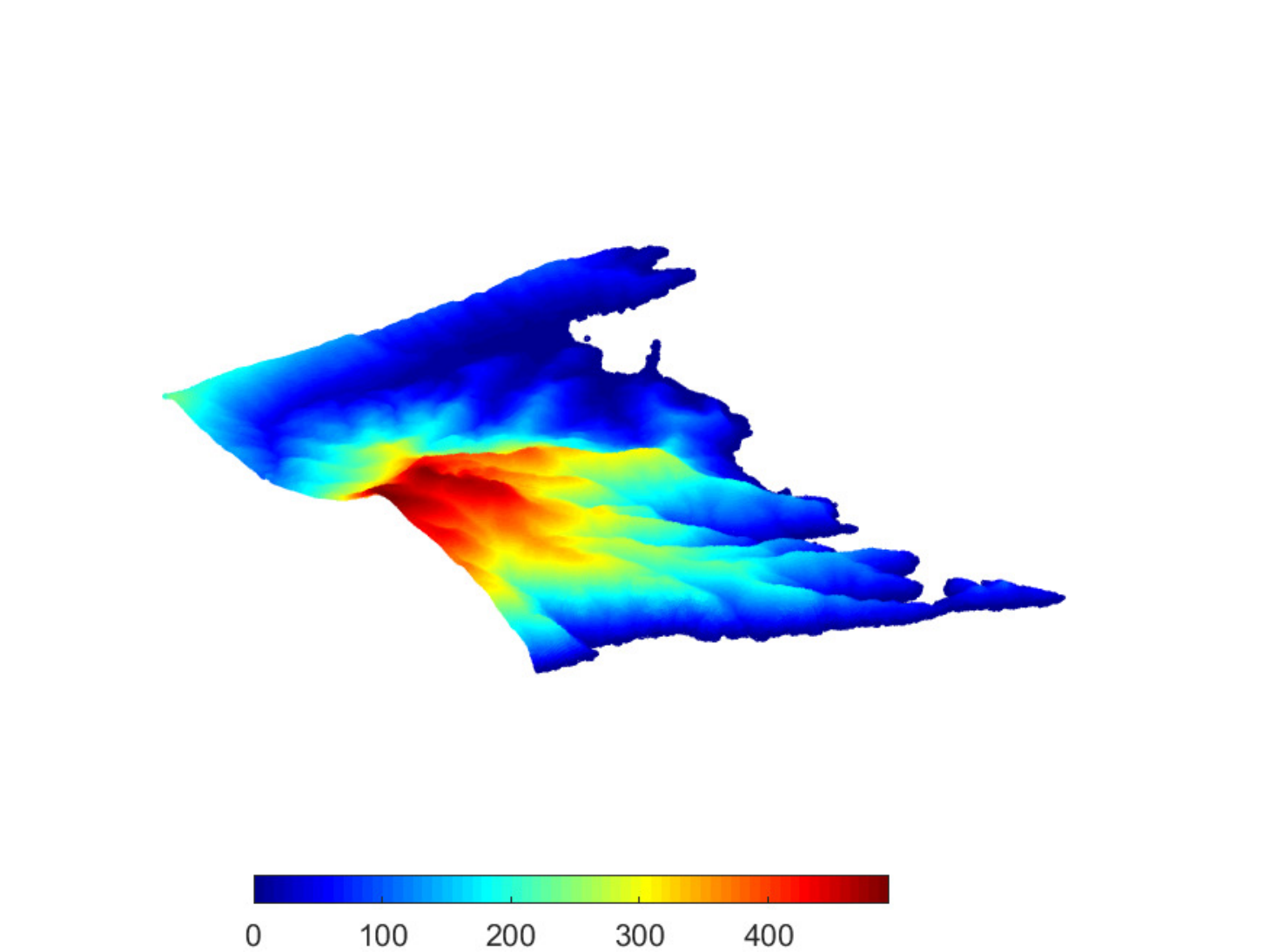} &
\includegraphics[scale=0.45]{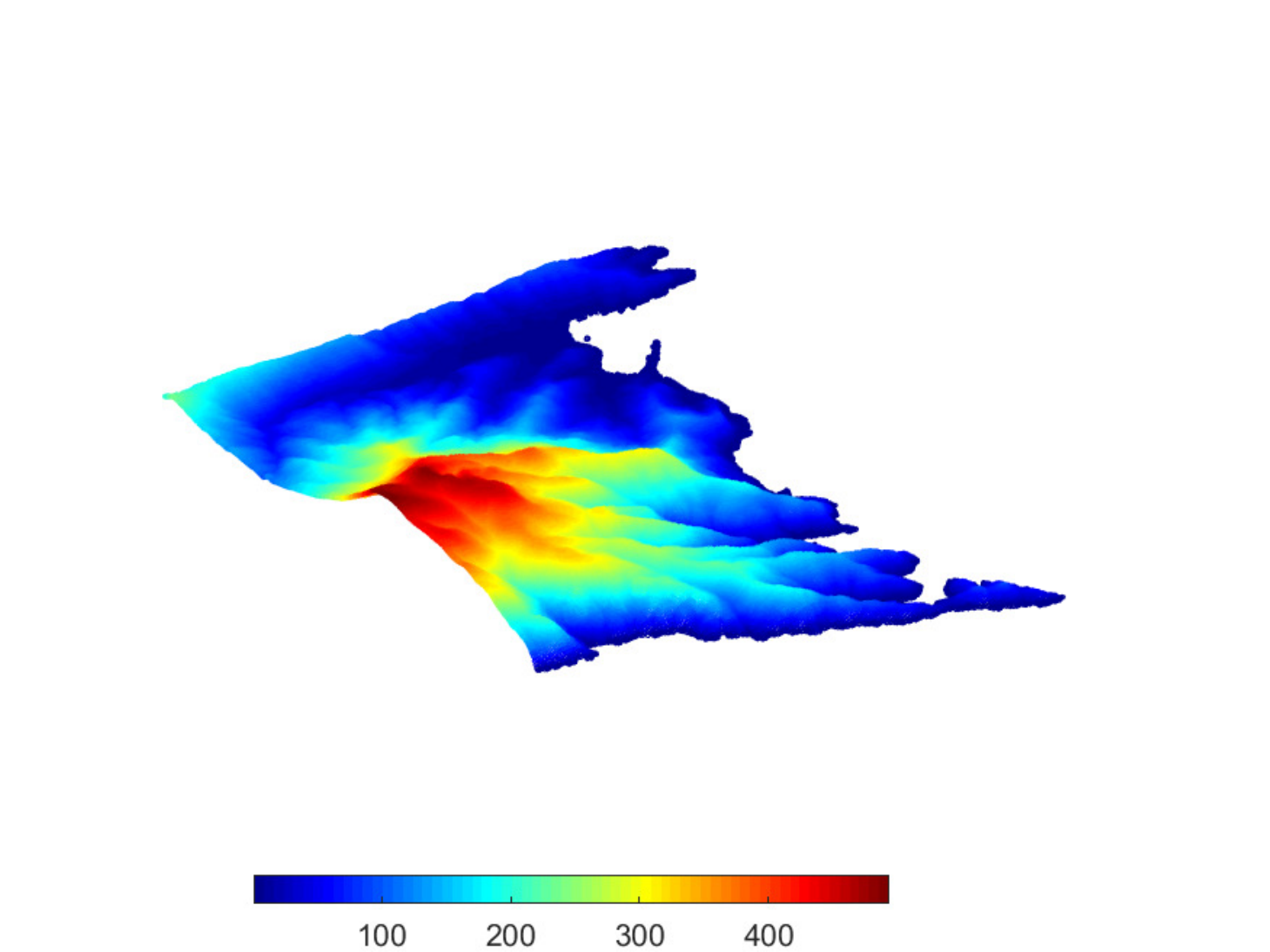} \\
(a) & (b)\\
\end{tabular}
\begin{tabular}{c}
 \includegraphics[scale=0.4, angle=0]{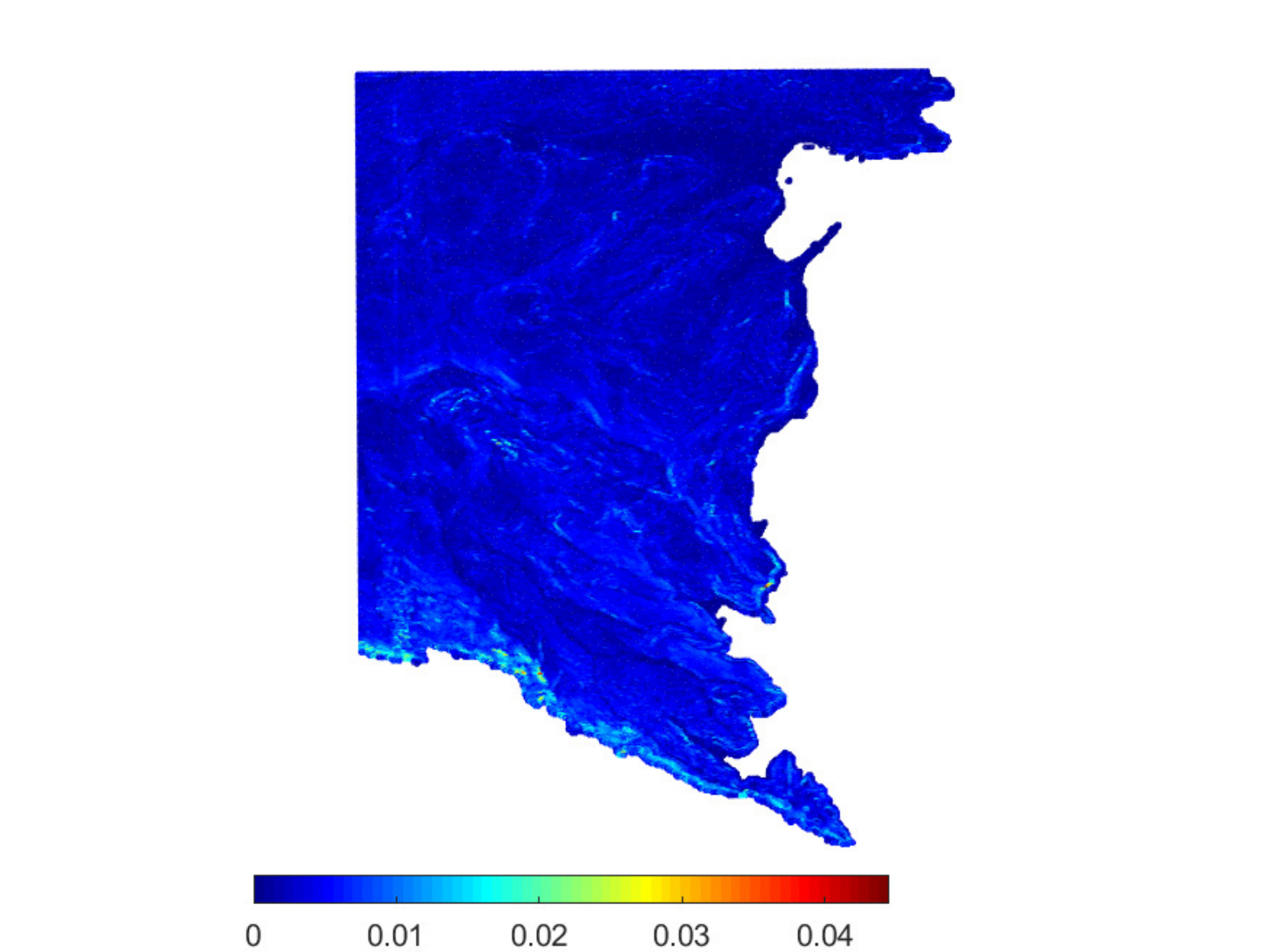} \\
(c)\\
\end{tabular}
\caption{Portofino, Liguria, Italy. A data point cloud from the given region of interest (a) is approximated via an IDW weight (b). The colors represent the elevation and vary from blue (low elevation) to red (high elevation). A graphical representation of the punctual error, normalized by the maximum elevation, is provided in (c). The statistics for the error are: min=$0.0000$, max=$0.0445$, mean=$0.0021$, median=$0.0017$, RMSE=$0.0029$ and std=$0.0019$.}
\label{figure:portofino}
\end{center}
\end{figure}

The method has been also tested for the approximation of the boundary of 3D models. As currently stated, wQISA is suitable to approximate surface portions that can be represented in a local Cartesian coordinate system in the form $z=f(x,y)$. Therefore, the object surface needs a subdivision into charts, for instance following the approaches in \cite{Ohtake:2003,SORGENTE2018}. Once the charts have been obtained, we compute the desired representation adopting as $z$ value the height value of the chart with respect to its best fitting regression plane \cite{Torrente2018}. Figure \ref{figure:vase} visually show some details of two wQISA approximations for 3D points clouds: the surfaces in the boxes approximate  the regions pointed by the (light blue) lines. These models come from the Visionair Shape Repository, VSR \cite{VISIONAIR}. Given the low level of noise, a pure $1$-NN weight function is here tested. The approximation shows a correct recovery of the main details of the artefact. Nevertheless, the feeble details are lost as an effect of the smoothing effect of this weight function.

\begin{figure}[!ht]
\begin{center}
\begin{tabular}{cc}
\includegraphics[scale=0.43]{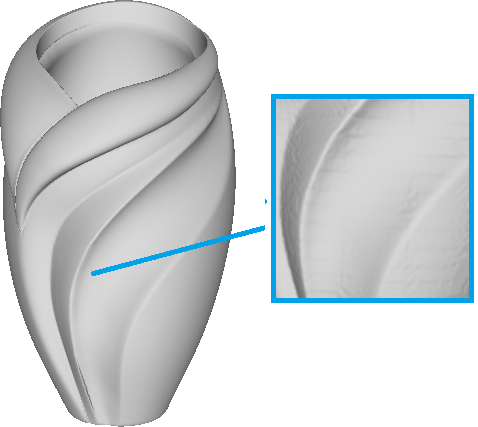} & \includegraphics[scale=0.4]{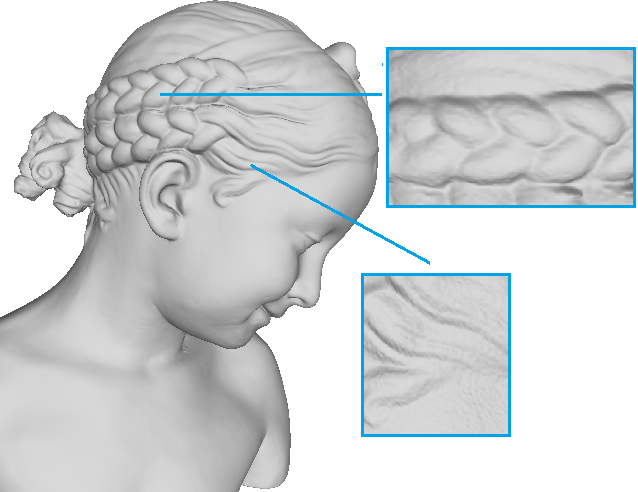}\\
(a) & (b)\\
\end{tabular}
\caption{Examples on two 3D models. For each model we highlight some details of the wQISA approximation computed by a $1$-NN weight function. The statistics of the relative punctual error are: for the vase, MSE=$4.1884e-06$ and std=$0.0015$; for the curl, MSE=$4.1493e-06$ and std=$0.0016$; for the tress, MSE=$3.7918e-05$ and std=$0.0057$.}
\label{figure:vase}
\end{center}
\end{figure}

\subsection{Approximation of surface properties}
As a further case study, we propose the approximation of a precipitation event over the Liguria region. To this purpose we consider an event occurred between January 16 and 20,~2014, which was responsible of heavy rain for about five days over all the Liguria region. The data we are considering were gathered from rain gauges maintained by Regione Liguria.  The network is spread over the whole region, with ~$143$ measure stations. These data come from the use case adopted for the comparison of six rainfall precipitation methods in \cite{PATANE_2017}.

Here, we compare wQISA with $k$-NN weight functions with two other methods: radial basis functions (RBF) with Gaussian kernel, as considered in \cite{PATANE_2017}, and the Multilevel B-splines Approximation (MBA) \cite{LeeMBA}. In the RBF implementation a global support \cite{Carr:2001} is adopted (all the 143 rain samples are considered) and a direct solver is applied to the linear system, which is symmetric and positive-definite. The MBA approximation is obtained with the default settings of the implementation of the Geometry Group at SINTEF Digital, which is freely available at: \url{https://github.com/orochi663/MBA}.

A quantitative comparison is provided in Table \ref{table:comparison_rains} and computed by performing $5$ times a $5$-fold cross-validation on each method. For more details, we refer once again the reader to \cite{Hastie2009} (chapter 7). The optimal parameters for a k-NN wQISA are chosen by minimizing the average MSE and are: $k=9$, with $10$ inner knots for each direction. In Figure \ref{fig:precipitation} we sample the precipitation fields approximated with the three methods in a set of points, representing the Liguria region. Although our approximation looks smoother and less detailed, it has in practice a better generalization performance as a learning method -- that is a better prediction capability on independent test data. An implementation of the wQISA method for rainfall data with the choice of the optimal values for the $k$ parameter and the cross-validation tests reported in Table~\ref{table:comparison_rains} is freely availale at: https://github.com/rea1991/wQISA.

\begin{table}
\begin{center}
\caption{Statistics for the error distribution of the cross validation.\label{table:comparison_rains}}
\begin{tabular}{m{0.5in} m{0.5in} m{0.5in} m{0.5in} m{0.5in} m{0.5in} m{0.5in}}
\midrule
Method & \vtop{\hbox{\strut Min}\hbox{\strut [mm]}} & \vtop{\hbox{\strut Max}\hbox{\strut [mm]}} & \vtop{\hbox{\strut Mean}\hbox{\strut [mm]}} & \vtop{\hbox{\strut Median}\hbox{\strut [mm]}} & \vtop{\hbox{\strut Std}\hbox{\strut [mm]}} & \vtop{\hbox{\strut MSE}\hbox{\strut [mm\textsuperscript{2}]}} \\[-2ex]
\midrule 
RBF   & 0.0317 & 2.9363 & 1.0903 & 1.0070 & 0.7830 & 1.7973 \\ 
MBA   & 0.0341 & 3.3489 & 1.1667 & 1.0243 & 0.8767 & 2.1969\\ 
wQISA & 0.0471 & 2.8293 & 0.9883 & 0.9013 & 0.6885 & 1.4657\\ 

\midrule
\end{tabular}
\end{center}
\end{table}

\begin{figure}[!ht]
\begin{center}
\begin{tabular}{S S S}
\includegraphics[scale=0.28,trim={4.1cm 3cm 1.25cm 2cm},clip]{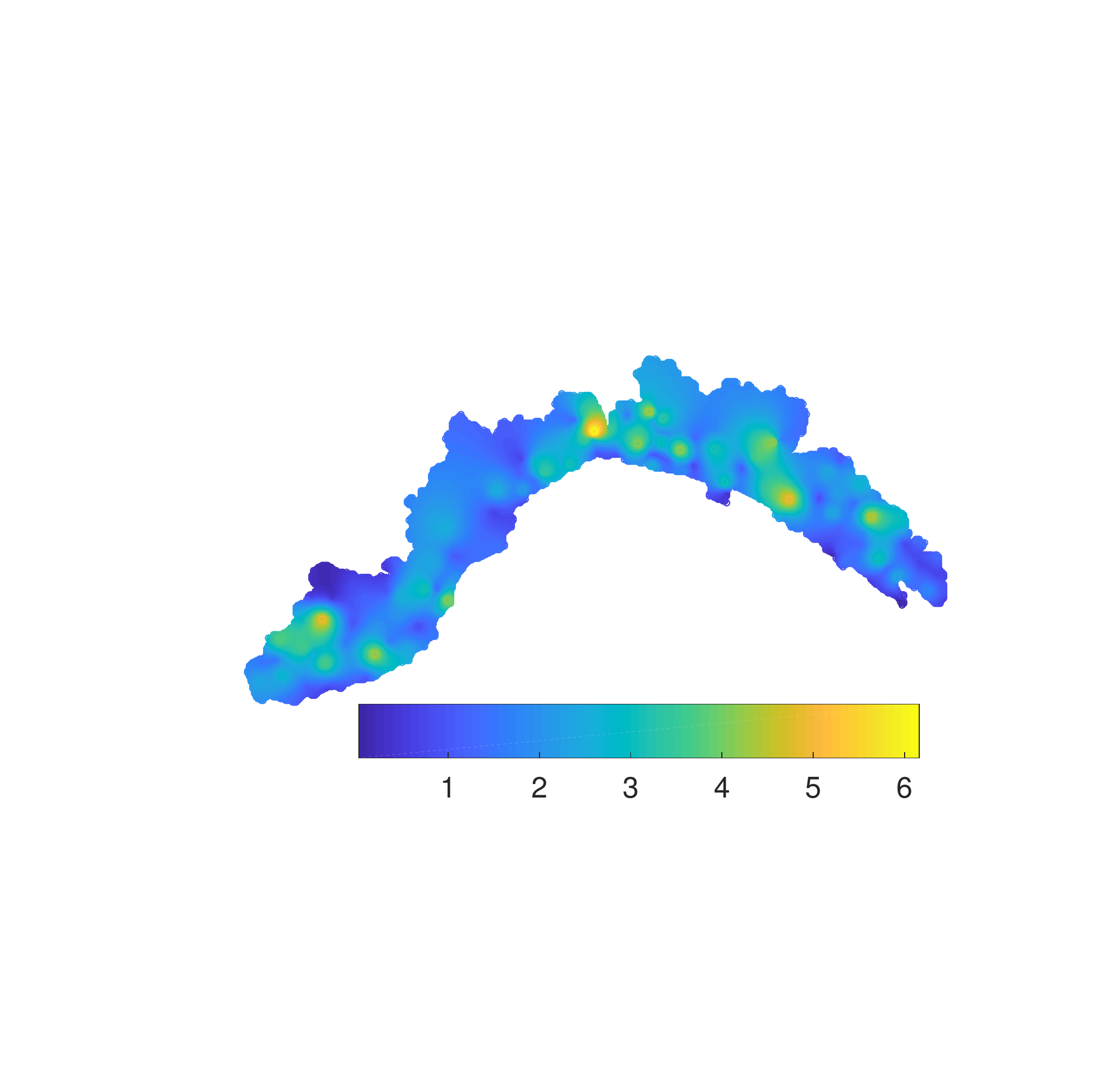}
&
\includegraphics[scale=0.28,trim={4.1cm 3cm 1.25cm 2cm},clip]{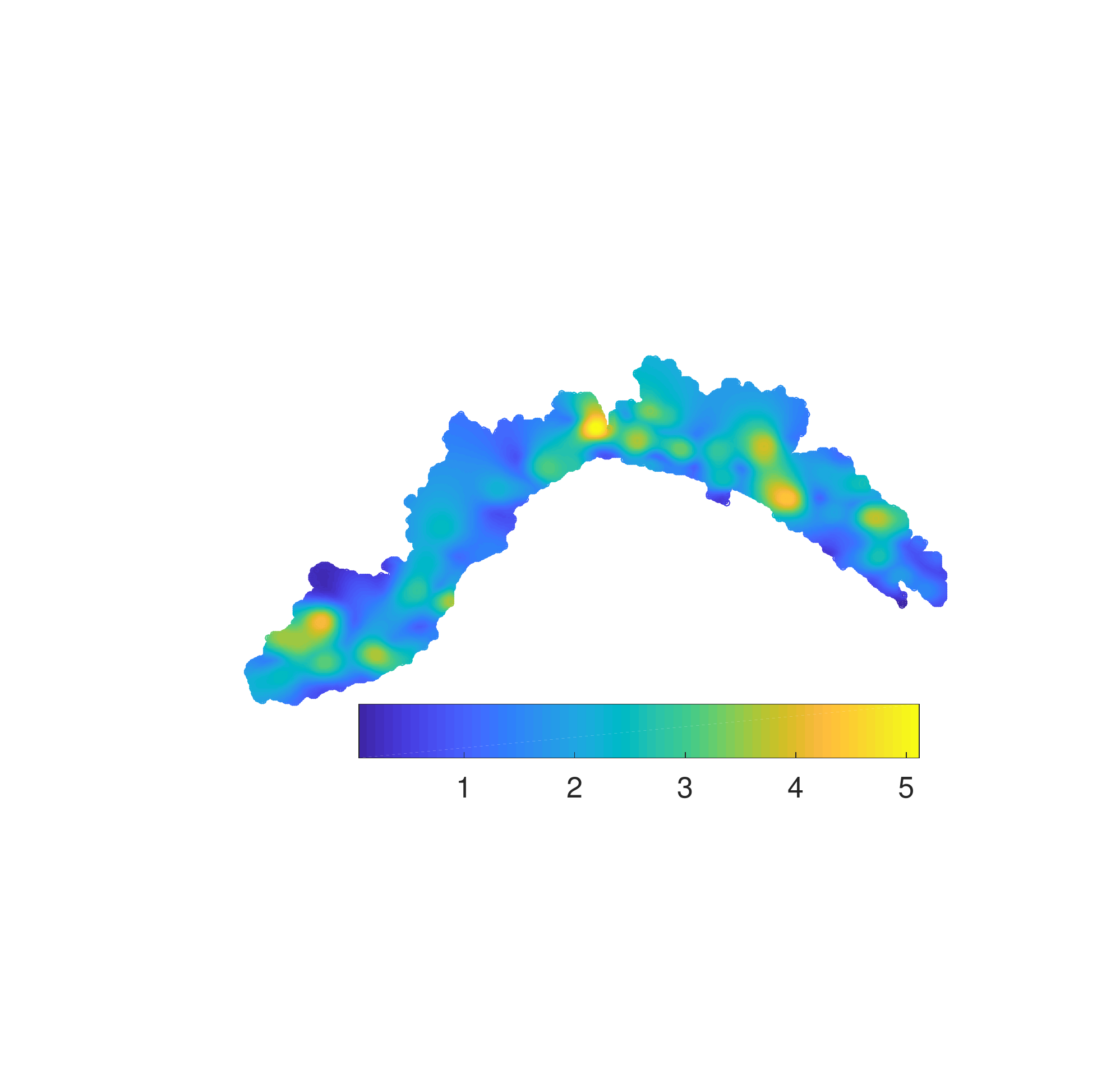}
&
\includegraphics[scale=0.29,trim={4.1cm 3cm 1.25cm 2cm},clip]{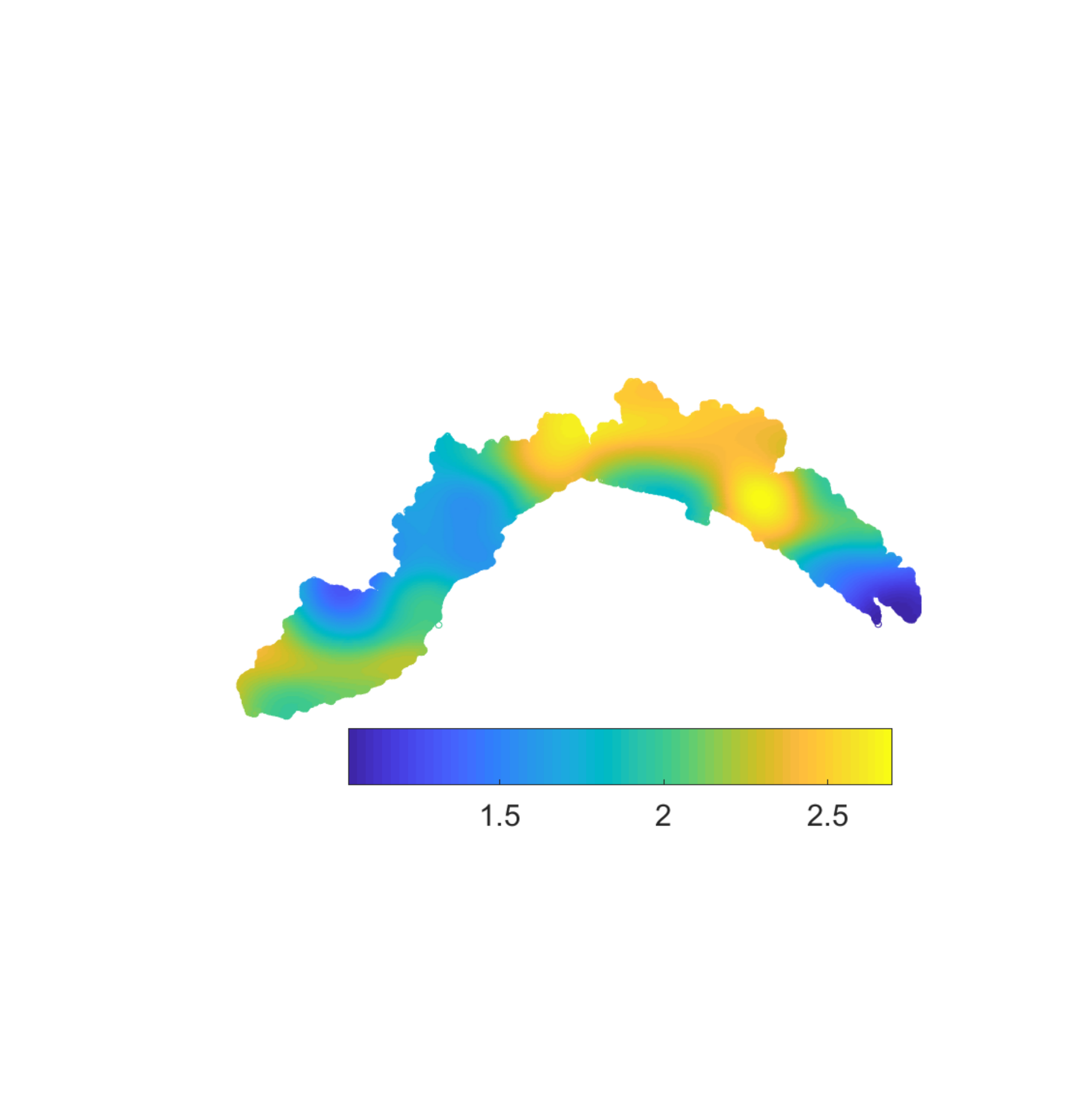} \\
(a) & (b) & (c)
\end{tabular}
\caption{Rainfall approximation with RBF (a), MBA (b) and wQISA (c).}
\label{fig:precipitation}
\end{center}
\end{figure}

\section{Concluding remarks and future perspectives}
\label{sec:conclusion}
We defined a novel quasi-interpolant reconstruction technique (wQISA), specifically designed to handle large and noisy point sets, even when equipped with outliers. 
The robustness and the versatility of the method are theoretically discussed from the point of view of numerical analysis (Sections \ref{sec:properties}) and probability theory (Section \ref{sec:prob_interp}). Numerical examples are provided in Section \ref{sec:simulations}. 

In this work we presented a quasi-interpolant scheme that applies to point clouds even equipped with noise and outliers. Our definition of the control point estimators combines computational efficiency with the possibility to work with different types of noise, as well as a reduced sensitivity to outliers. The computational complexity is, in fact, comparable to that of a weighted average. We gave evidence of the approximation effectiveness of the method over a wide range of real data and application domains.

As a further development of the method, we think it is possible to extend wQISA to more general refinement schemes, for instance opportunely selecting the point neighbours \cite{Lenarduzzi1993}, such as in the case of LR B-splines or THB-splines \cite{Dokken2013,Johannessen2015}. 
This is particularly relevant because these locally refining schemes naturally deal with isogeometric computations and simulation and offers the valuable perspective to practically adopt this work for Computer Aided Design and Manufacturing (CAD/CAM), Finite Element Analysis and IsoGeometric Analysis \cite{Johannessen2014,Giannelli2016,Occelli2019}.

\newpage
\section*{Acknowledgments}
\begin{sloppypar}
The authors thank Dr. Bianca Falcidieno and Dr. Michela Spagnuolo for the fruitful discussions; Dr. Oliver J. D. Barrowclough, Dr. Tor Dokken and Dr. Georg Muntingh for their concern as supervisors; the reviewers, for their positive suggestions, which have significantly contributed to extend the references.

Funding: this project has received funding from the European Union's Horizon 2020 research and innovation programme under the Marie Sk\l{}odowska-Curie grant agreement No 675789. This work has been co-financed by the ``POR FSE, Programma Operativo Regione Liguria" 2014-2020, No RLOF18ASSRIC/68/1, and
partially developed in the CNR-IMATI activities DIT.AD021.080.001 and DIT.AD009.091.001.

This is a pre-print of an article published in Numerical Algorithms. The final authenticated version is available online at: https://doi.org/10.1007/s11075-020-00989-4”.
\end{sloppypar}

\bibliographystyle{ieeetr}      
\bibliography{main.bbl}
\appendix

\newpage
\section{Univariate case}
\label{Appendix}

We will suppose -- up to a rotation -- that the point cloud $\mathcal{P}$ can be locally represented by a function of the form $f:[a,b]\subset\mathbb{R}\to\mathbb{R}$.
\begin{definition}
Let $\mathcal{P}\subset\mathbb{R}^2$ be a point cloud, $p\in\mathbb{N}^\ast$ and $\mathbf{x}=[x_1,\ldots,x_{n+p+1}]$ a $(p+1)$-regular (global) knot vector with fixed boundary knots $x_{p+1}=a$ and $x_{n+1}=b$. The \emph{Weighted Quasi Interpolant Spline Approximation} of degree $p$ to the point cloud $\mathcal{P}$ over the knot vector $\mathbf{x}$ is defined by
\begin{equation}
\label{equation:wQISA_1D}
f_w(x):=\sum_{i=1}^n\hat{y}_w(\xi^{(i)})B[\mathbf{x}^{(i)}](x),
\end{equation}
where $\xi^{(i)}:=(x_i+\ldots+x_{i+p})/p$ are the \emph{knot averages} and 
\[
\hat{y}_w(t):=\dfrac{\sum\limits_{(x,y)\in\mathcal{P}}y\cdot w_t(x)}{\sum\limits_{(x,y)\in\mathcal{P}}w_t(x)}
\]
are the \emph{control points estimators} of \emph{weight functions} $w_t:\mathbb{R}\to[0,+\infty)$.
\end{definition}

\subsection{Properties}

\subsubsection{Global and local bounds}
\begin{proposition}[Global bounds]
\label{proposition:global_bounds_1D}
Let $\mathcal{P}\subset\mathbb{R}^2$ be a point cloud. Given $y_{\text{min}}, y_{\text{max}}\in\mathbb{R}$  that satisfy
\[
y_{\text{min}}\le y\le y_{\text{max}}, \quad \text{ for all } (x,y)\in\mathcal{P},
\]
then the weighted quasi interpolant spline approximation to $\mathcal{P}$ from some spline space $\mathbb{S}_{p,\mathbf{x}}$ and some weight function $w$ has the same bounds
\[
y_{min}\le f_w(x)\le y_{max}, \quad \text{ for all } x\in\mathbb{R}.
\]
\end{proposition}
\begin{proof}
From the partition of unity property of a B-spline basis, it follows that
\begin{equation}
\label{equation:global_bounds_1D}
\begin{array}{@{}c@{\;}c@{\;}c@{\;}c@{\;}c@{}}
\min_i\hat{y}_w(\xi^{(i)}) & \le & \sum_{i=1}^n\hat{y}_w(\xi^{(i)})B[\mathbf{x}^{(i)}](x) & \le & \max_i\hat{y}_w(\xi^{(i)}) \\
\vge\,\scriptsize{\circled{1}} && \vdequal  && \vle\,\scriptsize{\circled{2}} \\
y_{min} &&f_w(x) && y_{max}
\end{array}
\end{equation}
where the inequalities $\circled{1}$ and $\circled{2}$ are a direct consequence of defining $\hat{y}_w$ by means of a convex combination.
\end{proof}
The bounds of Proposition \ref{proposition:global_bounds_1D} can potentially lead to local bounds. We discuss this situation in Corollary \ref{corollary:local_bounds_1D}.

\begin{corollary}[Local bounds]
\label{corollary:local_bounds_1D}
Let $\mathcal{P}\subset\mathbb{R}^2$ be a point cloud. If $x\in[x_\mu,x_{\mu+1})$ for some $\mu$ in the range $p+1\le\mu\le n$, then
\[
\alpha(\mu)\le f_w(x)\le \beta(\mu),
\]
for some $\alpha(\mu),\beta(\mu)$ which belong to $ [y_{min},y_{max}]$.
\end{corollary}
\begin{proof}
By using the property of local support for B-splines, it follows that
\[
f_w(x)=\sum_{i=\mu-p}^\mu\hat{y}_w(\xi^{(i)})B[\mathbf{x}^{(i)}](x)
\]
over $[x_\mu,x_{\mu+1})$. Thus, we can re-write the chain of inequalities \eqref{equation:global_bounds_1D} as
\begin{equation}
\label{equation:local_bounds_1D}
\begin{array}{@{}c@{\;}c@{\;}c@{\;}c@{\;}c@{}}
\min\limits_{i=\mu-p,\dots,\mu}\hat{y}_w(\xi^{(i)}) & \le & f_w(x) & \le & \max\limits_{i=\mu-p,\dots,\mu}\hat{y}_w(\xi^{(i)}) \\
\vge\,\scriptsize{\circled{3}} &&   && \vle\,\scriptsize{\circled{4}} \\
\min\Big\{y \text{ s.t. }(x,y)\in{\textstyle\bigcup\limits_{i=\mu-p}^\mu}\mathcal{P}_i\Big\} &&  && \max\Big\{y \text{ s.t. }(x,y)\in{\textstyle\bigcup\limits_{i=\mu-p}^\mu}\mathcal{P}_i\Big\}
\\
\vdequal &&   && \vdequal \\
\alpha(\mu) && && \beta(\mu)
\end{array}
\end{equation}
where 
\[\mathcal{P}_i:={\textstyle\bigcup\limits_{i=\mu-p,\ldots,\mu}}
\left\{\text{supp}\left(w_{\xi^{(i)}}(\cdot)\right)
\right\}\cap\mathcal{P}.\]
\end{proof}
Notice that the set of points which are effectively used to compute the approximation, i.e., 
\[
\mathcal{P}^\ast:=\bigcup\limits_{i=p+1,\ldots,n}\mathcal{P}_i
\]
may be a proper subset of $\mathcal{P}$.

\subsubsection{Preservation of monotonicity}

\begin{definition}[$w$-monotonicity]
Let $w_t:\mathbb{R}\to[0,+\infty)$ be a family of weight functions, where $t\in\mathbb{R}$. A point cloud $\mathcal{P}\subset\mathbb{R}^2$ is said to be $w$-\emph{increasing} if for all $x_1\le x_2$, $\hat{y}_w(x_1)\le\hat{y}_w(x_2)$. $\mathcal{P}$ is said to be $w$-\emph{decreasing} if for all $x_1\le x_2$, $\hat{y}_w(x_1)\ge\hat{y}_w(x_2)$.
\end{definition}

\begin{figure}[htp]
\begin{center}
\begin{tabular}{cc}
\includegraphics[scale=0.38]{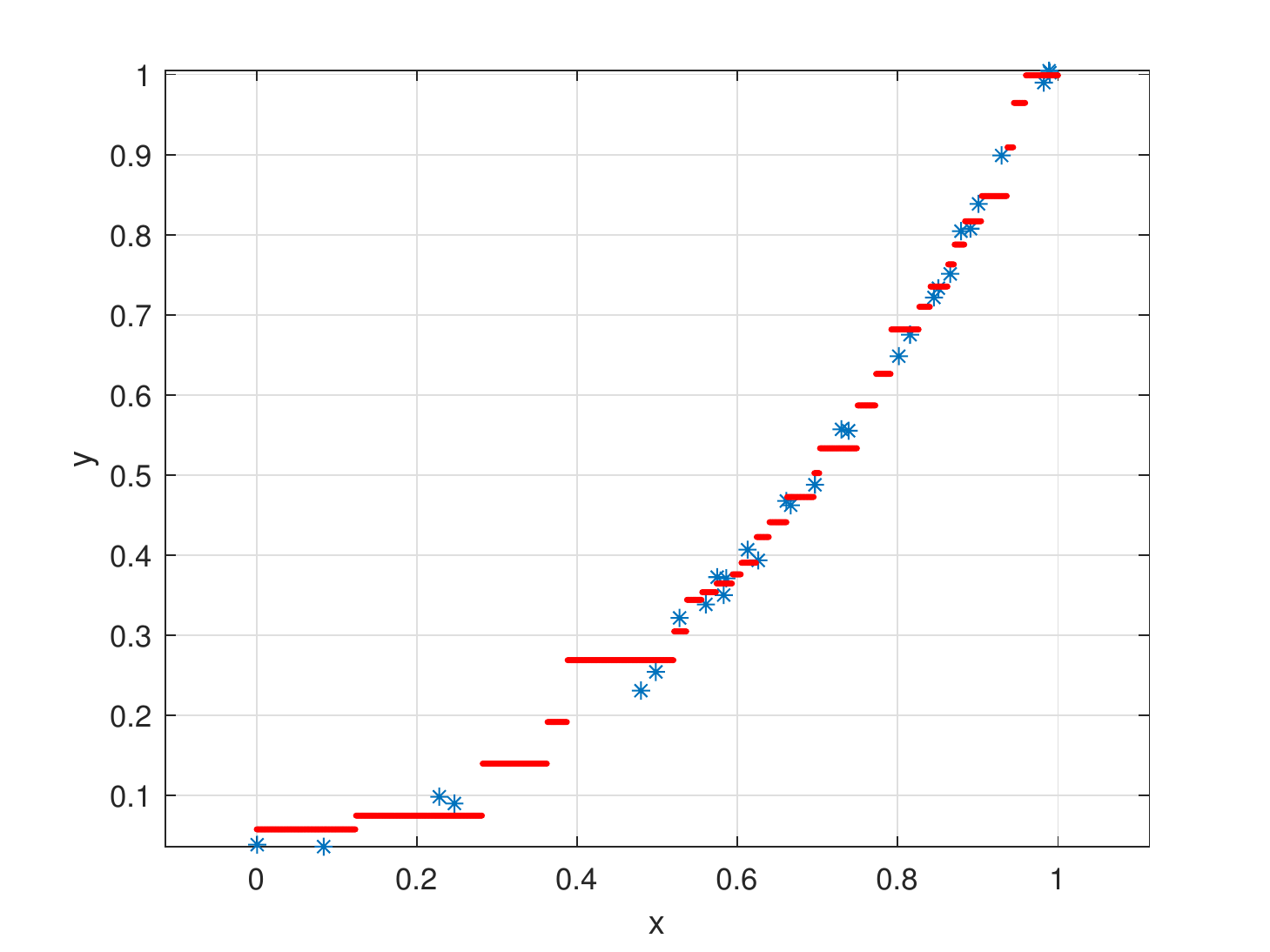} 
&
\includegraphics[scale=0.38]{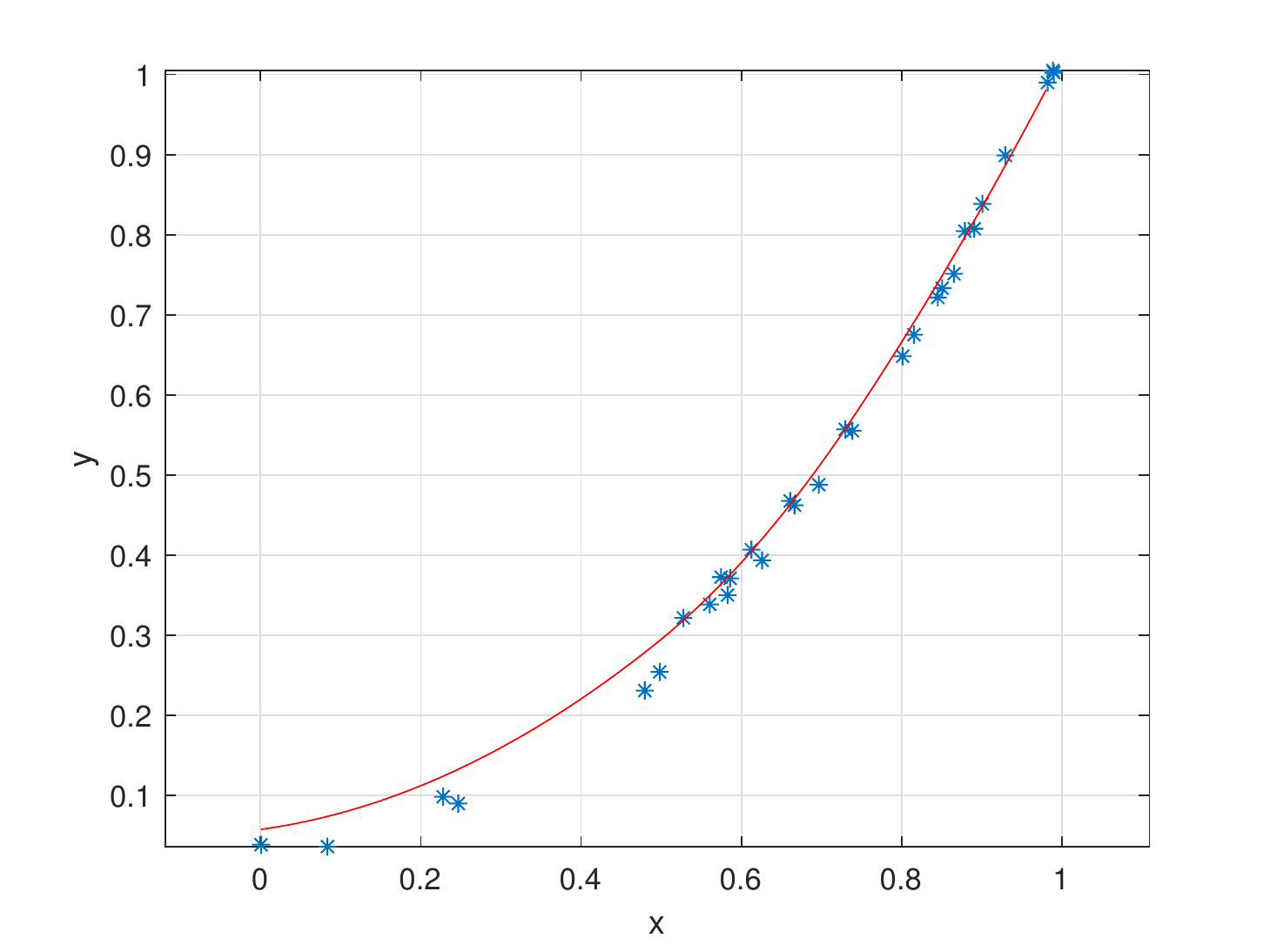} \\
(a) & (b) 
\end{tabular}
\caption{$w$-monotonicity and its preservation. Figure (a) shows an example of an estimator $\hat{y}_w:\mathbb{R}\to\mathbb{R}$ (in red) for a given point cloud (in blue) with respect to a 3-NN weight function. Figure (b) graphically compares the original point cloud (in blue) to its wQISA (in red).}
\label{figure:monotonicity}
\end{center}
\end{figure}

The key ingredient to prove the preservation of monotonicity through our method is the following lemma.
\begin{lemma}
\label{lemma:monot_coeff_1D}
Let  $p\in\mathbb{N}^\ast$ and $\mathbf{x}=[x_1,\ldots,x_{n+p+1}]$ be a $(p+1)$-regular (global) knot vector with fixed boundary knots $x_{p+1}=a$ and $x_{n+1}=b$. In addition, let $f=\sum_{i=1}^nc_iB[\mathbf{x}^{(i)}]\in\mathbb{S}_{p,\mathbf{x}}$. If the sequence of coefficients $\{c_i\}_{i=1}^n$ is increasing (decreasing) then $f$ is increasing (decreasing).
\end{lemma}
\begin{proof}
The Lemma is proven in \cite{Lyche2011}, pp. 114--115.
\end{proof}
\begin{proposition}
\label{proposition:uni_pres_mon}
Let $\mathcal{P}\subset\mathbb{R}^2$ be a point cloud, $p\in\mathbb{N}^\ast$ and $\mathbf{x}=[x_1,\ldots,x_{n+p+1}]$ be a $(p+1)$-regular (global) knot vector with fixed boundary knots $x_{p+1}=a$ and $x_{n+1}=b$. If $\mathcal{P}$ is $w$-increasing (decreasing) then $f_w$ is also increasing (decreasing).
\end{proposition}
\begin{proof}
By definition of $w$-increasing (decreasing) point cloud, the sequence of control points $\{\hat{y}_w(\xi^{(i)})\}_{i=1}^n$ is increasing (decreasing). By Lemma \ref{lemma:monot_coeff_1D}, this is sufficient to conclude that $f_w$ is increasing (decreasing).
\end{proof}

\subsubsection{Preservation of convexity}
\begin{definition}[$w$-convexity]
Let $w_t:\mathbb{R}\to[0,+\infty)$ be a family of weight functions, where $t\in\mathbb{R}$. A point cloud $\mathcal{P}\subset\mathbb{R}^2$ is said to be $w$-\emph{convex} if for all $x_1\le x_2$ and for any $\lambda\in[0,1]$, \[\hat{y}_w((1-\lambda)x_1+\lambda x_2)\le(1-\lambda)\hat{y}_w(x_1)+\lambda\hat{y}_w(x_2).\] $\mathcal{P}$ is said to be $w$-\emph{concave} if $\mathcal{P}_{-}:=\{(x,-y) | (x,y)\in\mathcal{P}\}$ is $w$-convex.
\end{definition}
\begin{figure}[htp]
\begin{center}
\includegraphics[scale=0.38]{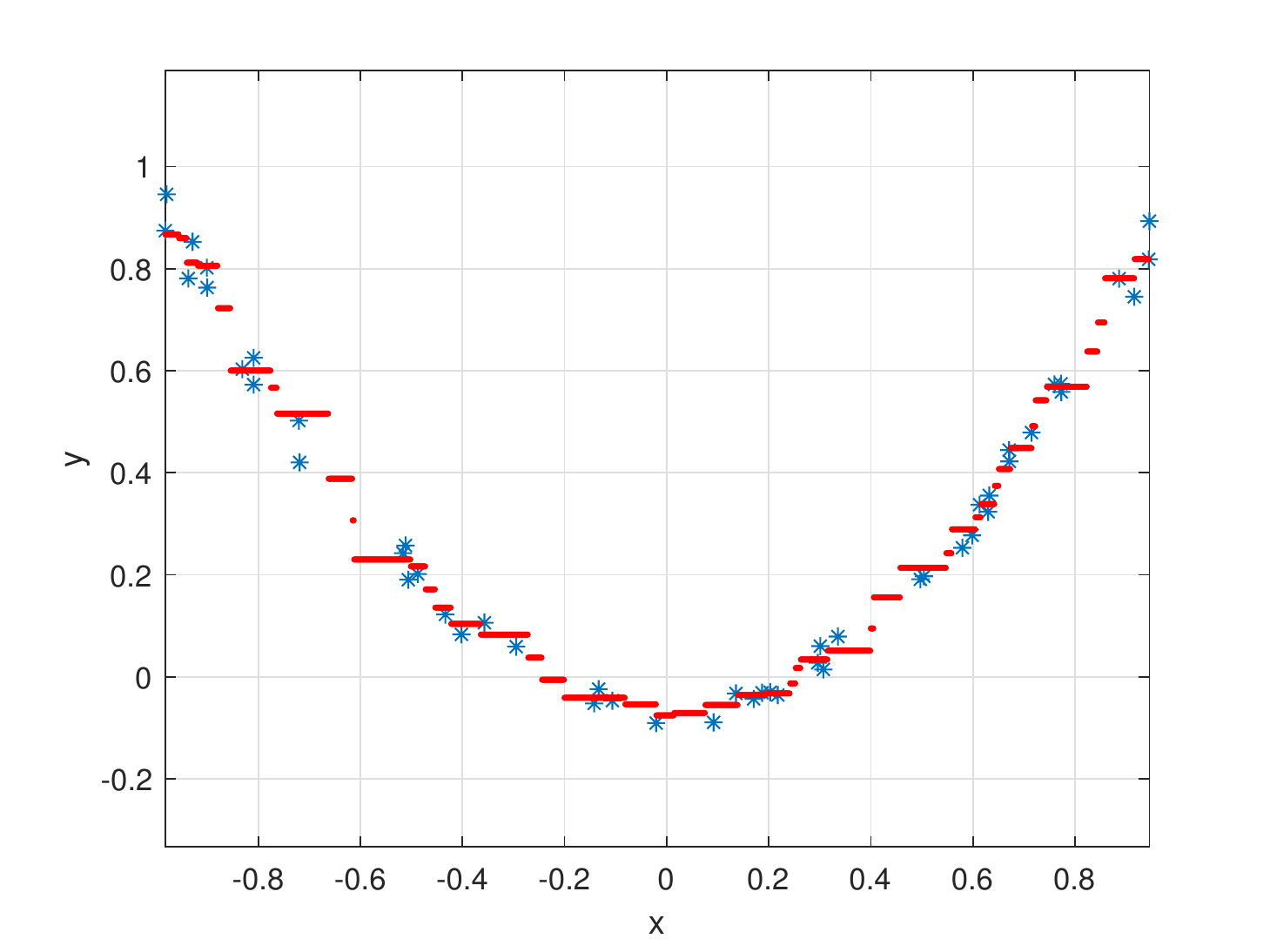}
\includegraphics[scale=0.38]{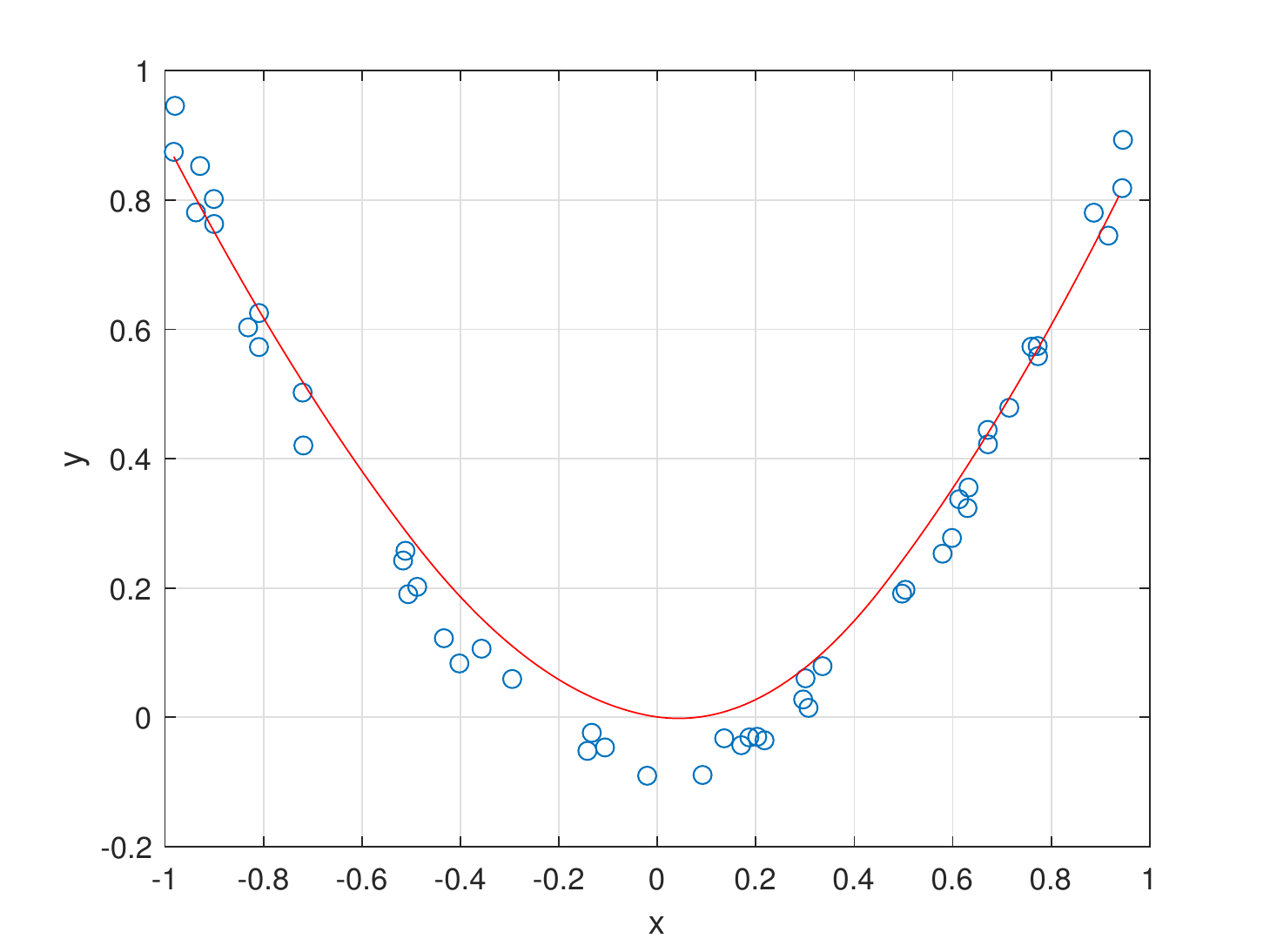}
\caption{$w$-convexity and its preservation. Figure (a) shows an example of an estimator $\hat{y}_w:\mathbb{R}\to\mathbb{R}$ (in red) for a given point cloud (in blue) with respect to a 3-NN weight function. Figure (b) graphically compares the original point cloud (in blue) to its wQISA (in red).}
\label{figure:convexity}
\end{center}
\end{figure}
The preservation of convexity is a consequence of the following lemma.
\begin{lemma}
\label{lemma:convex_coeff_1D}
Let  $p\in\mathbb{N}^\ast$ and $\mathbf{x}=[x_1,\ldots,x_{n+p+1}]$ be a $(p+1)$-regular (global) knot vector with fixed boundary knots $x_{p+1}=a$ and $x_{n+1}=b$. Lastly, let $f=\sum_{i=1}^nc_iB[\mathbf{x}^{(i)}]\in\mathbb{S}_{p,\mathbf{x}}$. Define $\Delta c_i$ by
\[
\Delta c_i:=
\begin{cases}
\dfrac{c_i-c_{i-1}}{x_{i+p}-x_i}, 	&\text{ if }x_i<x_{i+p} \\
\Delta c_{i-1}					&\text{ if }x_i=x_{i+p}
\end{cases}
\]
for $i=2,\ldots,n$. Then $f$ is convex on $[x_{p+1},x_{n+1}]$ if it is continuous and if the sequence $\{\Delta c_i\}_{i=2}^n$ is increasing.
\end{lemma}
\begin{proof}
See \cite{Lyche2011}, p. 118.
\end{proof}
\begin{proposition}
\label{proposition:uni_pres_conv}
Let $\mathcal{P}\subset\mathbb{R}^2$ be a point cloud, $p\in\mathbb{N}^\ast$ and $\mathbf{x}=[x_1,\ldots,x_{n+p+1}]$ be a $(p+1)$-regular (global) knot vector with fixed boundary knots $x_{p+1}=a$ and $x_{n+1}=b$. If $\mathcal{P}$ is $w$-convex (concave) then $f_w$ is also convex (concave).
\end{proposition}
\begin{proof}
Let
\[
\Delta c_i:=
\dfrac{\hat{y}_w(\xi^{(i)})-\hat{y}_w(\xi^{(i-1)})}{x_{i+p}-x_i}=\dfrac{\hat{y}_w(\xi^{(i)})-\hat{y}_w(\xi^{(i-1)})}{(\xi^{(i)}-\xi^{(i-1)})p}
\]
with  $x_{i}<x_{i+p}$. Since $\mathcal{P}$ is $w$-convex then these differences must be increasing and consequently $f_w$ is convex by Lemma \ref{lemma:convex_coeff_1D}.
\end{proof}

\end{document}